\documentclass[11pt,a4paper]{article}

\usepackage[left=2.45cm, top=2.45cm,bottom=2.45cm,right=2.45cm]{geometry}
\usepackage{tikz}
\usetikzlibrary{arrows.meta}
\usepackage{mathtools,amssymb,amsthm,mathrsfs,calc,graphicx,xcolor,cleveref,dsfont,tikz,pgfplots,bm, extarrows}
\usepackage[british]{babel}
\usepackage{amsfonts}             
\usepackage[T1]{fontenc}
\numberwithin{equation}{section}

\newtheorem {theorem}{Theorem}[section]

\newtheorem {lemma}[theorem]{Lemma}
\newtheorem {corollary}[theorem]{Corollary}

{\theoremstyle{definition}

}
{\theoremstyle{theorem}

}

\newcommand{\one}{\mathbf{1}}
\newcommand{\E}{\mathbf{E}}
\newcommand{\p}{\mathbf{P}}
\newcommand{\Var}{\mathbf{Var}}
\newcommand{\N}{\mathbb{N}} 
\newcommand{\R}{\mathbb{R}} 

\usepackage{subcaption}

\usepackage{enumitem}

\def\cA{\mathcal{A}}
\def\cB{\mathcal{B}}
\def\cC{\mathcal{C}}

\def\cE{\mathcal{E}}

\def\cG{\mathcal{G}}

\def\cI{\mathcal{I}}

\def\cM{\mathcal{M}}

\def\cP{\mathcal{P}}

\def\cW{\mathcal{W}}
\def\cX{\mathcal{X}}

\makeatletter
\let\@fnsymbol\@alph
\makeatother
\allowdisplaybreaks

\begin{document}

\title{\bfseries Large components in the subcritical \\ Norros-Reittu model}

\author{Matthias Lienau\footnotemark[1]\;\; and Matthias Schulte\footnotemark[2]}

\date{}
\renewcommand{\thefootnote}{\fnsymbol{footnote}}

\footnotetext[1]{Hamburg University of Technology. Email: matthias.lienau@tuhh.de}

\footnotetext[2]{Hamburg University of Technology. Email: matthias.schulte@tuhh.de}

\maketitle

\begin{abstract}The Norros-Reittu model is a random graph with $n$ vertices and i.i.d.\ weights assigned to them. The number of edges between any two vertices follows an independent Poisson distribution whose parameter is increasing in the weights of the two vertices. Choosing a suitable weight distribution leads to a power-law behaviour of the degree distribution as observed in many real-world complex networks. We study this model in the subcritical regime, i.e.\ in the absence of a giant component. For each component, we count the vertices and show convergence of the corresponding point process to a Poisson process. More generally, one can also count only specific vertices per component, like leaves. From this one can deduce asymptotic results on the size of the largest component or the maximal number of leaves in a single component. The results also apply to the Chung-Lu model and the generalised random graph.\vspace*{15pt}	
\end{abstract}

\section{Introduction and main results}
Complex networks are very large graphs with a difficult structure, which arise in a lot of different fields, ranging from biology and computer science over epidemiology to sociology. One can think of the brain, (tele-)communication networks, the internet or even root systems of trees, to name some  more explicit examples. Even though these graphs appear in very different contexts, surprisingly they often share a couple of properties. One of them is the  scale-free behaviour, which is also called a power-law. This means that the proportion of vertices having degree $k$ is approximately proportional to $k^{-\beta-1}$ for some $\beta>0$, resulting in the proportion of vertices having degree at least $k$ being approximately proportional to $k^{-\beta}$, which can be weakened to $k^{-\beta}\ell(k)$ for some slowly varying function $\ell$. While the first way of phrasing might be more intuitive, the latter allows for a description in terms of the tail of the typical degree distribution of the graph.  See \cite{vdH_book} for detailed information on complex networks and the survey article \cite{welldone} for the scale-free property. Due to the interest in complex networks there has been a lot of attention on finding random graph models which show a scale-free behaviour and ideally also share other properties of real-world complex networks. One of these random graphs is the Norros-Reittu model introduced in \cite{NorrosReittu2006}, which we study in this paper. Note that our results are for a regime where the erased Norros-Reittu model obtained by deleting multiple edges and loops is asymptotically equivalent to similar models such as the Chung-Lu model from \cite{ChungLu, ChungLu2} or the generalised random graph from \cite{GenRandGraph} as shown in Example 3.6 in \cite{Janson2008}. This class of models is also referred to as rank-1 inhomogeneous random graphs.

In this paper we consider the subcritical regime in which no giant component exists. In every component we count the vertices. We study the collection of these counting statistics and prove convergence to a Poisson process. Our framework also allows to count instead of all vertices in a component only vertices of a certain type, such as leaves or vertices of a fixed degree.

The rest of this paper is organised as follows. In the remainder of this section we present our main results after introducing the required notation. All proofs are postponed to Section \ref{sec:2}.

We consider a sequence of (multi-)graphs $(G_n)_{n\in\N}$ given by the Norros-Reittu model, which is defined as follows. Let $\cW=(W_i)_{i\in\N}$ be a sequence of independent copies of a positive random variable $W$. For $n\in\N$, one takes $[n]= \{1,\dots,n\}$ as vertex set of $G_n$ and assigns the weights $W_1,\dots,W_n$ to the vertices. For $x,y\in[n]$, the number of edges between $x$ and $y$ in $G_n$ is given by the $\N_0$-valued random variable $E_n\{x,y\}$, where the random variables $(E_n\{x,y\})_{1\leq x\leq y\leq n}$ are, conditionally on $\cW$, independent with
\begin{align*}
	E_n\{x,y\}\sim\text{Poisson}(W_xW_y/L_n)\quad\text{for}\quad L_n=\sum_{i=1}^nW_i.
\end{align*}  
In particular, the Norros-Reittu model allows multiple edges and loops, i.e.\ edges from a vertex to itself. We require the following assumptions on the weight distribution throughout this paper:
\begin{align*}
	\mathrm{(W)}\quad  &\text{The distribution of the positive random variable $W$ has a regularly}\\
	&\text{varying tail with exponent $\beta>2$ and $\E[W^2]<\E[W]$.}
\end{align*}
The first part of the previous assumption means that 
\begin{align*}
	\p(W>t)= t^{-\beta} L(t)\quad \text{for}\quad t>0,
\end{align*}
where $L:(0,\infty)\to(0,\infty)$ is a slowly varying function, i.e.\ for all $c>0$,
\begin{align*}
	\lim_{t\to\infty}\frac{L(ct)}{L(t)}=1.
\end{align*}
For more details on regular variation and slowly varying functions we refer to  Chapter 2 in \cite{Resnick2007} or to \cite{reg_var}. Note that $\beta>2$  in (W) implies the existence of the second moment of $W$.

The degree of any vertex follows a $\mathrm{Poisson}(W)$ distribution, i.e.\ a mixed Poisson distribution. The regularly varying tail of $W$ ensures that the graph is scale-free in the sense that the degree distribution has a regularly varying tail with exponent $\beta$, which can be proven along the lines of Corollary 13.1 in \cite{BolJanRio2007}. The condition $\E[W^2]<\E[W]$ corresponds to the subcritical regime concerning the phase transition of the component sizes. In \cite{BolJanRio2007} it has been shown in Theorem 3.1, see also Section 16.4, that the graph exhibits a giant component, that is a component which contains a positive fraction of all vertices, if and only if $\E[W^2]>\E[W]$ - this is also called the supercritical regime or supercritical phase. The maximal component size in the case $\E[W^2]=\E[W]$, which is called the critical regime, has been addressed in \cite{vdHBhavLe2010} and \cite{vdHBhavLe2012}. It is shown in \cite{vdHBhavLe2010} that if $\beta>3$, which yields a finite third moment of $W$, the largest component is of order $n^{2/3}$ whereas \cite{vdHBhavLe2012} provides the order $n^{(\beta-1)/\beta}$ for $\beta\in(2,3)$, where $W$ has an infinite third moment. The latter result requires the additional assumption that $L(t)$ converges to a constant as $t\to\infty$. For the subcritical regime where $\E[W^2]<\E[W]$ it is shown in \cite{Janson2008_comp_size} that the order of the largest component is at most $n^{1/\beta}$ and related to the maximal degree of the graph, under the assumption that $\p(W>t)=O(t^{-\beta})$. We complement this result by deriving the asymptotic distribution.

We define $q(n)=F^{-1}(1-1/n)$ for $n\in\N$  where $F^{-1}$ denotes the quantile function of the distribution function $F$ of $W$. For two vertices $x,y\in[n]$ we write $x\sim y$ if $x$ and $y$ are connected via a path in $G_n$ and denote the component of $x$ by $\cC_n(x)= \{z\in[n]\colon x\sim z\}$. Note that $x\sim x$ for all $x\in[n]$ by using the convention that a path may consist of just a single vertex. Throughout the paper we denote by $\p_\cW$ and $\E_\cW$ the conditional probability and expectation when conditioning on the weights $\cW=(W_i)_{i\in\N}$. For all $n\in\N$ we consider a function $(v,G_n)\mapsto \cX_n(v)\subseteq [n]$ where $v\in[n]$ and assume that $\cX_n(v)\subseteq \cC_n(v)$. For $n\in\N$ and $v\in[n]$, $\cX_n(v)$ is the class of vertices we would like to count in the component of $v$ in $G_n$. We write $S_n(v)$ for the cardinality of $\cX_n(v)$ and require the following two assumptions.

(A1) \quad There exists some $\xi>0$ such that
\begin{align*}
	\frac{1}{q(n)}\underset{v=1,\dots,n}{\sup}\left|\E_\cW[S_n(v)]-W_{v}\xi\right|\overset{\p}{\longrightarrow} 0\quad \text{as}\quad n\to\infty.
\end{align*}

(A2) \quad Let $n\in\N, v\in[n]$ and $x\in\cC_n(v)$. Checking whether $x\in\cX_n(v)$ only depends on all \hspace*{51pt} paths that start in $v$ and contain $x$.

In this paper, thus also in assumption (A2), we denote by a path a sequence of distinct vertices $v_1\dots v_k$ such that $v_i$ and $v_{i+1}$ are connected by at least one edge for $i\in[k-1]$. Consequently, assumption (A2) implies that whether $x\in\cX_n(v)$ is satisfied cannot depend on loops or multiple edges.

Finally, we choose from every component in $G_n$ one vertex with the largest weight. If the largest weight is not unique, we choose the one with the smallest label. We denote the collection of these vertices by $V^{\max}_n\subseteq [n]$. We study the point processes
\begin{align*}
	\Xi_n=\sum_{v=1}^n\mathbf{1}\{v\in V^{\max}_{n}\}\delta_{S_n(v)q(n)^{-1}\xi^{-1}}
\end{align*}
for $n\in\N$, where the constant $\xi>0$ is specified in assumption (A1). Here, $\delta_x$ stands for the Dirac measure concentrated at $x\in\mathbb{R}$ and we think of point processes as random counting measures. The indicator demanding the weight of $v$ being maximal in its component ensures that we only count each component once in $\Xi_n$.

Our main theorem shows point process convergence in $M_p((0,\infty])$, the space of point measures on $(0,\infty]$. We will not discuss this space in detail but instead refer to \cite{Resnick2007}, Chapter 3.

\begin{theorem}\label{Thm:conv_to_PP}
	For $n\in\N$, consider the Norros-Reittu model with weights satisfying assumption $\mathrm{(W)}$ and $(\cX_n(v))_{v\in[n]}$ such that assumptions $\mathrm{(A1)}$ and $\mathrm{(A2)}$ hold. Then
	\begin{align}
		\Xi_n= \sum_{v=1}^n\mathbf{1}\{v\in V^{\max}_{n}\}\delta_{S_n(v)q(n)^{-1}\xi^{-1}}\overset{d}{\longrightarrow}\eta_\beta\quad\text{as}\quad n\to\infty,\label{eq:main_thm}
	\end{align}
	where $\eta_\beta$ is a Poisson process with intensity measure $\nu_\beta((a,b])=a^{-\beta}-b^{-\beta}$ for all $0<a<b\leq\infty$. 
\end{theorem}
Note that the points of $\Xi_n$ at zero are not taken into account in \eqref{eq:main_thm} as we work on the space of point processes on $(0,\infty]$. The key idea to prove Theorem \ref{Thm:conv_to_PP} is to approximate $S_n(v)$ by $W_v\xi$ and show that $\Xi_n$ is close to $\Theta_n=\sum_{v=1}^n\delta_{W_vq(n)^{-1}}$, the collection of rescaled weights, which converges to $\eta_\beta$ as $n\to\infty$. A similar strategy was already employed in \cite{Matthias_Chinmoy} to study large degrees of some random graphs. However, our analysis is more involved as the considered statistics of the components are less local than degrees.

From the point process convergence in Theorem \ref{Thm:conv_to_PP} we deduce the following asymptotic behaviour of the maximum of $S_n(1),\dots,S_n(n)$.
\begin{corollary}\label{cor:conv_max}
	Under the same assumptions as in Theorem \ref{Thm:conv_to_PP}
	\begin{align*}
		\frac{1}{q(n)\xi}\underset{v=1,\dots,n}{\max}S_n(v)\overset{d}{\longrightarrow} Z_\beta\quad\text{as}\quad n\to\infty,
	\end{align*}
	where $Z_\beta$ is a random variable following a Fr\'echet distribution with parameter $\beta$.
\end{corollary}
Corollary \ref{cor:conv_max} follows immediately from Theorem \ref{Thm:conv_to_PP}. The distribution function of the left-hand side in the corollary is given by void probabilities of the intervals $(t,\infty], t\in\R,$ of the underlying point process $\Xi_n$. Since these intervals are relatively compact in the usual topology on $(0,\infty]$, we can pass on to the limiting process $\eta_\beta$, see e.g.\ Theorem 3.2 in \cite{Resnick2007}, to obtain the result.

In the next theorem we provide some classes of vertices that satisfy the assumptions of Theorem \ref{Thm:conv_to_PP}. For two vertices $x,y\in[n]$ of $G_n$ let ${\rm{d}}_{G_n}(x,y)$ denote the graph distance between $x$ and $y$ which is the number of edges of the shortest path between $x$ and $y$ and let $\deg_{G_n}(x)$ stand for the degree of $x$ which is the number of neighbours of $x$. Since for us a path is a sequence of distinct vertices connected by edges, condition (A2) is only satisfied for the number of neighbours of a vertex $x$ and not for the number of edges incident to $x$. However, our result would not change if the decision whether $x\in\cX_n(v)$ also depends on the multiple edges and self-loops involving $x$. Indeed, one can show that the expected number of loops and of vertex pairs connected by more than one edge are bounded. Thus, the way we count vertices with multiple edges or loops is asymptotically irrelevant as $n\to\infty$ due to the rescaling by $q(n)$ in \eqref{eq:main_thm}.

\begin{theorem}\label{Lem:assump_for_different_vertices}
	Under assumption $\mathrm{(W)}$, the following classes of vertices satisfy \eqref{eq:main_thm} with $\xi$ as stated below. 
	\begin{enumerate}
		\item \textbf{All vertices:} $\cX_n(v)=\cC_n(v)$ with $\xi=\frac{\E[W]}{\E[W]-\E[W^2]}$.
		\item \textbf{Vertices in a fixed distance} $m\in\N$ \textbf{to $v$}: $\cX_n(v)=\{x\in\cC_n(v)\colon {\rm{d}}_{G_n}(x,v)=m\}$ with $\xi=\left(\frac{\E[W^2]}{\E[W]}\right)^{m-1}$.		
		\item \textbf{Vertices with fixed degree $m\in\N$}: $\cX_n(v)=\{x\in\cC_n(v)\colon \deg_{G_n}(x)=m\}$ with $\xi=\frac{1}{(m-1)!}\frac{\E[W^me^{-W}]}{\E[W]-\E[W^2]}$.
		\item \textbf{Terminal trees:} Let $m\in\N$ and consider a rooted tree $T$ with vertex set $V(T)=[m]$ and  root $1$. For $x\in[n]$ we say that $x\in\cX_n(v)$ if and only if there is exactly one path from $v$ to $x$ and the subgraph generated by $x$ and its descendants with $x$ as root is isomorphic to $T$, ignoring loops and multiple edges. We have
		\[\xi=\frac{1}{c(T)}\frac{\E[W^{\deg_T(1)+1}e^{-W}]}{\E[W]-\E[W^2]}\prod_{i=2}^m\frac{\E[W^{\deg_T(i)}e^{-W}]}{\E[W]},\] where $c(T)$ is the order of the automorphism group of $T$ that preserves the root and $\deg_T$ denotes the degree of a vertex in the tree $T$.
	\end{enumerate}
\end{theorem}
We wish to add some remarks on the classes of vertices in Theorem \ref{Lem:assump_for_different_vertices}. Counting the vertices in distance one of the vertex with largest weight provides its degree so that one studies the point process of degrees of vertices having maximal weight in their component. In that case, our result yields the same limiting process as obtained in \cite{Matthias_Chinmoy}, where the point process of all degrees was investigated. 

In \cite{Janson2008_comp_size} the author shows in a similar setting that the $k$-th largest component size satisfies $|\cC_{(k)}|=\xi d_{(k)} + o_p(q(n))$ where $d_{(k)}$ denotes the $k$-th largest degree of the graph. The idea is to do a breadth-first exploration of the component, starting in the vertex having the largest degree, which is similar in spirit to our approach. If one puts together the results from \cite{Janson2008_comp_size} and \cite{Matthias_Chinmoy}, one can also conclude the first claim of Theorem \ref{Lem:assump_for_different_vertices}. 

The fourth class of vertices above is motivated by the fact that the large components are tree-like as $n\to\infty$. We are able to count the number of subgraphs which are isomorphic to a given rooted tree and are "terminal" in the sense that they only spread away from the vertex $v\in V_n^{\max}$ (in the graph distance sense) and have no further edges attached to them. See also Figure 1 for a visualisation where the tree $T$ is a wedge with the vertex of degree $2$ as root. In the picture one can see a tree with hub $v$ in which there are two vertices that give birth to a terminal wedge, $v_3$ and $v_9$. The vertex $v$ itself does not qualify although $v,v_1,v_2$ form a wedge, because there are further edges. Similarly, the vertex $v_4$ is a root of the wedge $v_4,v_8,v_9$, which is not terminal since $v_9$ has more neighbours than just $v_4$. The subgraph induced by $v,v_2,v_5$ with $v_2$ as root forms also a wedge, but is not terminal since the wedge spreads from its root $v_2$ towards $v$.
\begin{figure}
	\centering
	\scalebox{0.8}{
		\begin{tikzpicture}
			[
			level 1/.style = {black, sibling distance = 4cm},
			level 2/.style = {black, sibling distance = 2.5cm}
			]
			
			\node [circle,draw] at (2,0) {$v$} 
			child {node[circle,draw] {$v_1$}}
			child {node[circle,draw] {$v_2$}
				child{node[circle,draw] {$v_{5}$}}}
			child {node[circle, ultra thick,draw] {$v_3$}
				child {node[circle,draw] {$v_6$}}
				child {node[circle,draw] {$v_7$}}}
			child {node[circle,draw] {$v_4$}
				child {node[circle,draw] {$v_8$}}
				child {node[circle, ultra thick,draw] {$v_9$}
					child {node[circle,draw] {$v_{10}$}}
					child {node[circle,draw] {$v_{11}$}}}
			}

			;

	\end{tikzpicture}}
	\caption{Counting terminal wedges in $\cC_n(v)$, here with $v_3$ and $v_9$ as roots.}
\end{figure}
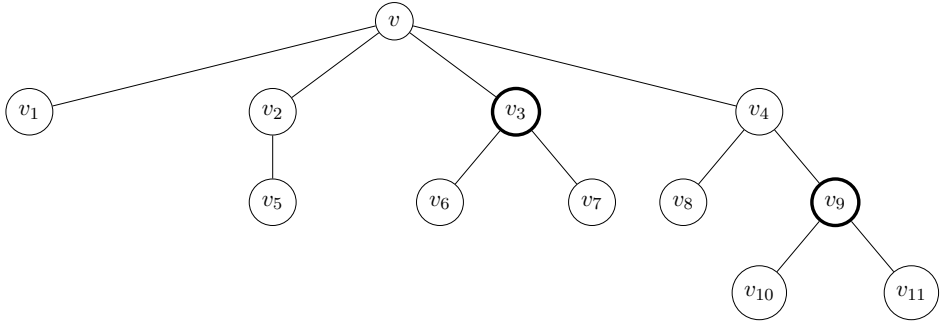

Finally, it is possible to study vertices of a certain degree or the terminal trees from the fourth example in a fixed distance $m\in\N$ to $v$, as in the second example. This of course affects the choice of $\xi$.

We conclude this section with a transfer of the results to related models. Write $p_{ij}=W_iW_j/L_n$ for $i,j\in[n]$ and $n\in\N$. In the Norros-Reittu model denoted by ${\rm{NR}}(n)$ two distinct vertices $i,j\in[n]$ are connected by (at least) one edge with probability 
\begin{align}
	\p_\cW(i\leftrightarrow j)=1-\exp(-p_{ij}),\label{eq:nr}
\end{align}
conditionally on the weights $\cW$. We write ${\rm{ENR}}(n)$ for the corresponding erased Norros-Reittu model without loops or multiple edges. There are a couple of similar models, for example the Chung-Lu model presented by \cite{ChungLu,ChungLu2} denoted by ${\rm{CL}}(n)$, where 
\begin{align}
	\p_\cW(i\leftrightarrow j)=\min(1,p_{ij})\label{eq:cl},
\end{align}
or the generalised random graph model considered by \cite{GenRandGraph} we call ${\rm{GRG}}(n)$ with 
\begin{align}
	\p_\cW(i\leftrightarrow j)=\frac{p_{ij}}{p_{ij}+1}\label{eq:grg}
\end{align}
for distinct vertices $i,j\in[n]$. Note that we consider both ${\rm{CL}}(n)$ and ${\rm{GRG}}(n)$ with slight modifications from their original form as in \cite{vdEsker} (see also Chapter 6 in \cite{vdH_book}). Example 3.6 from \cite{Janson2008} shows that the three random graph models ${\rm{ENR}}(n),{\rm{CL}}(n)$ and ${\rm{GRG}}(n)$ are asymptotically equivalent, given that $\p(W>t)=o(t^{-2})$ -- this tail behaviour is ensured by assumption (W). Asymptotic equivalence means that there exists a coupling of the graphs such that the probability of them being non-isomorphic to each other tends to zero as $n\to\infty$. Because both ${\rm{CL}}(n)$ and ${\rm{GRG}}(n)$ provide simple graphs, it is important to consider ${\rm{ENR}}(n)$ instead of ${\rm{NR}}(n)$. 
\begin{theorem}\label{Thm:NR_CL_GRG} Under assumption {\rm{(W)}} the following hold.
	\begin{enumerate}
		\item Let $(\cX_n(v))_{v\in[n],n\in\N}$ be identical for the Norros-Reittu model ${\rm{NR}}(n)$ and its erased version ${\rm{ENR}}(n)$, i.e.\ $\cX_n(v)$ does not take loops or multiple edges into account. If \eqref{eq:main_thm} is valid for the Norros-Reittu model ${\rm{NR}}(n)$, then it also holds for the Chung-Lu model ${\rm{CL}}(n)$ and the generalised random graph ${\rm{GRG}}(n)$.
		\item For all classes of vertices considered in  Theorem 1.3, \eqref{eq:main_thm} also holds for the Chung-Lu model ${\rm{CL}}(n)$ and the generalised random graph ${\rm{GRG}}(n)$.
	\end{enumerate}
\end{theorem}
Note that all choices of $(\cX_n(v))_{v\in[n]}$ we can consider in Theorem \ref{Thm:conv_to_PP} do not take loops and multiple edges into account and are thus identical for ${\rm{NR}}(n)$ and ${\rm{ENR}}(n)$.

There exists another slight variation of ${\rm{ENR}}(n),{\rm{CL}}(n)$ and ${\rm{GRG}}(n)$, namely when one replaces $p_{ij}$ in \eqref{eq:nr}, \eqref{eq:cl} and \eqref{eq:grg} by $p'_{ij}=W_iW_j/(n\E[W])$. We denote the corresponding models by ${\rm{ENR}}'(n),{\rm{CL}}'(n)$ and ${\rm{GRG}}'(n)$.  Note that some authors also use $\tilde{p}_{ij}=W_iW_j/n$ instead of $p'_{ij}$, but these models are related by rescaling the weights. Example 3.1 in \cite{Janson2008} shows that ${\rm{ENR}}'(n),{\rm{CL}}'(n)$ and ${\rm{GRG}}'(n)$ are also asymptotically equivalent. From the strong law of large numbers one obtains that the ratio $p_{ij}/p'_{ij}$ converges almost surely to $1$ as $n\to\infty$ which gives reason to believe that the models ${\rm{ENR}}(n)$ and ${\rm{ENR}}'(n)$ behave similarly. However, a quick calculation shows that Corollary 2.13 in \cite{Janson2008} does not yield asymptotic equivalence of ${\rm{ENR}}(n)$ and ${\rm{ENR}}'(n)$. However, Corollary 2.12 in \cite{Janson2008} provides contiguity of these two models which means that sequences of events $(\cA_n)_{n\in\N}$ whose probability tends to zero for ${\rm{ENR}}(n)$ also tends to zero for ${\rm{ENR}}'(n)$ and vice versa.

Our results also transfer from ${\rm{ENR}}(n)$ to the models ${\rm{ENR}}'(n),{\rm{CL}}'(n)$ and ${\rm{GRG}}'(n)$. We would like to sketch two ways how one can see this. On the one hand, one can repeat all calculations of this paper and replace $p_{ij}$ by $p'_{ij}$. On the other hand, we prove Theorem \ref{Thm:conv_to_PP} by showing \eqref{eq:pp_conv_on_a_infty}. This statement prevails under contiguity. Note that also the assumption {\rm{(A1)}} remains valid when switching between ${\rm{ENR}}(n)$ and ${\rm{ENR}}'(n)$ due to contiguity, while {\rm{(A2)}} does not even depend on the underlying random graph. Therefore, Theorem \ref{Lem:assump_for_different_vertices} remains true for ${\rm{ENR}}'(n)$. The conclusions then transfer to ${\rm{CL}}'(n)$ and ${\rm{GRG}}'(n)$ as these models are asymptotically equivalent to ${\rm{ENR}}'(n)$.

\section{Proofs}\label{sec:2}
\subsection{Proof of Theorem \ref{Thm:conv_to_PP}}
Before going to the graph-related proofs we need a fact about the asymptotic behaviour of the maximum of the first $n$ weights $W_1,\dots,W_n$ under assumption (W). It is part of a well-known result from extreme value theory called Fisher-Tippet-Gnedenko theorem, see e.g.\  Theorem 1.2.1 and Corollary 1.2.4 in \cite{DeHaan_Ferreira}. For $n\in\N$ write $W_{(n)}=\max_{i=1,\dots,n}W_i$. It holds that
\begin{align}
	\frac{W_{(n)}}{q(n)}\overset{d}{\longrightarrow}Z_\beta\quad\text{as}\quad n\to\infty,\label{W(n):scaling}
\end{align}
where $Z_\beta$ is a random variable following a Fr\'echet distribution with parameter $\beta$.  For the scaling factor $q(n)$ we obtain from Lemma 3.3 in \cite{Matthias_Chinmoy} that
\begin{align}
	q(n)=F^{-1}(1-1/n)=n^{1/\beta}\ell(n) \label{q(n):scaling}
\end{align}
for $n\in\N$ with some slowly varying function $\ell$.
In particular it follows that $n^{-1/2}W_{(n)}$ converges in probability to zero as $n\to\infty$, due to $\beta>2$ from assumption (W) and the fact that the slowly varying function $\ell(n)$ grows slower than any positive power of $n$, see Proposition $1.3.6\,(v)$ in \cite{reg_var}. Combined with the weak law of large numbers we obtain
\begin{align}
	\frac{W_{(n)}^2}{L_n}=\frac{W_{(n)}^2}{n}\frac{1}{n^{-1}L_n}\overset{\p}{\longrightarrow}0\quad \text{as}\quad n\to\infty. \label{eq:Wn_Ln}
\end{align}

Next, we introduce notation and discuss convergence of often occurring series. For a set $A$ we denote by $A^k_{\neq}$ the set of all $k$-tuples of pairwise distinct elements of $A$. Recall that $\E_\cW$ denotes the conditional expectation with respect to the weights. Naturally, all equalities and inequalities only hold $\p$-almost surely which we will not write explicitly in the following to keep the notation simple.  We typically approximate the number of vertices in the set $\cX_n(v_0)$ for $v_0\in[n]$ by counting paths originating in $v_0$ and leading to vertices in $\cX_n(v_0)$, resulting in a sum of the following kind,
\begin{align}
	T_{n}(v_0)=\sum_{k=1}^n\sum_{(v_1,\dots,v_k)\in([n]\setminus\{v_0\})^k_{\neq}}\prod_{i=1}^k\one\{v_i\leftrightarrow v_{i-1}\}\one\{v_k\in\cX_n(v_0)\},\label{eq:T_nv}
\end{align}
where $x\leftrightarrow y$ means that the vertices $x$ and $y$ are connected by at least one edge. We first sum over the length of the path starting in $v_0$ and then over its vertices, which we demand to be distinct. Writing $L_n^{(2)}=\sum_{v=1}^nW_v^2$ and using the upper bound one for the last indicator in \eqref{eq:T_nv}, conditional independence of the remaining indicators and $\one\{x\leftrightarrow y\}\leq E_n\{x,y\}$ for $x,y\in[n]$ leads to
\begin{align*}
	\E_\cW[T_n(v_0)]&\leq \sum_{k=1}^n\sum_{(v_1,\dots,v_k)\in[n]^k_{\neq}}\prod_{i=1}^k\E_\cW[E_n\{v_i,v_{i-1}\}]=\sum_{k=1}^n\sum_{(v_1,\dots,v_k)\in[n]^k_{\neq}}\prod_{i=1}^k\frac{W_{v_i}W_{v_{i-1}}}{L_n}\\
	&\leq W_{v_0}\sum_{k=1}^n\sum_{v_1,\dots,v_k=1}^n \prod_{i=1}^{k-1}\frac{W_{v_i}^2}{L_n}\frac{W_{v_k}}{L_n}\leq W_{v_0}\sum_{k=0}^n\bigg(\frac{L_n^{(2)}}{L_n}\bigg)^{k}\eqqcolon W_{v_0}S_{n,\cW}.
\end{align*}
From the strong law of large numbers it follows that
\begin{align*}
	\frac{L_n^{(2)}}{L_n}=\frac{n^{-1}L_n^{(2)}}{n^{-1}L_n}\overset{a.s.}{\longrightarrow}\frac{\E[W^2]}{\E[W]}\quad\text{as}\quad n\to\infty.
\end{align*}
The geometric sum formula together with $\E[W^2]<\E[W]$ by assumption (W) yields
\begin{align}
	S_{n,\cW}=\frac{1-\left(L_n^{(2)}/L_n\right)^{n+1}}{1-L_n^{(2)}/L_n}\overset{a.s.}{\longrightarrow}\frac{\E[W]}{\E[W]-\E[W^2]} \quad \text{as}\quad n\to\infty.\label{conv:Sw}
\end{align}
This argument also provides almost sure convergence for related expressions of the form
\begin{align}
	\sum_{k=0}^n p(k)\bigg(\frac{L_n^{(2)}}{L_n}\bigg)^{k},\label{conv:rel_Sw}
\end{align}
where $p$ is an arbitrary polynomial of finite degree.

For $n\in\N, x,y\in[n]$ with $x\neq y$, $\,A\subseteq [n]$ and $k\geq2$ we write
\begin{align*}
	\cP^{(n)}_k(x,y,A)&=\{(v_1,\dots,v_k)\in[n]^k_{\neq}\colon v_1=x, v_k=y, v_2,\dots,v_{k-1}\notin A, v_1\leftrightarrow\hdots\leftrightarrow v_k\},\\
	\cP^{(n)}_k(x,A)&=\{(v_1,\dots,v_k)\in[n]^k_{\neq}\colon v_1=x, v_2,\dots,v_{k}\notin A, v_1\leftrightarrow\hdots\leftrightarrow v_k\}.
\end{align*}
The first set consists of all $k$-tuples of vertices of $G_n$ that form a path with endpoints $x$ and $y$ whose inner vertices do not lie in $A$. Similarly, the second set consists of all $k$-tuples which form a path whose starting vertex is $x$ and all other vertices do not belong to $A$. We are often interested in bounding the expected number of such paths, conditionally on the weights, where we can use the following lemma.
\begin{lemma}\label{Lem:paths}
	Let $x,y\in[n]$ with $x\neq y$, $\,A\subseteq [n]$ and $k\geq2$. Then
	\begin{align*}
		\E_\cW[|\cP^{(n)}_k(x,y,A)|]\leq \frac{W_xW_y}{L_n}\bigg(\frac{L_n^{(2)}}{L_n}\bigg)^{k-2} \quad\text{and}\quad\E_\cW[|\cP^{(n)}_k(x,A)|]\leq W_x\bigg(\frac{L_n^{(2)}}{L_n}\bigg)^{k-2}.
	\end{align*}
\end{lemma}
\begin{proof}
	We only prove the first assertion, the second one follows immediately from the first by summing over $y\in[n]\setminus\{x\}$ and using $\sum_{y\in[n]\setminus\{x\}}W_y\leq L_n$. As all edges in the paths are distinct and therefore conditionally independent we obtain
	\begin{align*}
		&\E_\cW[|\cP^{(n)}_k(x,y,A)|]\leq \E_\cW\bigg[\sum_{(v_1,\dots,v_k)\in[n]^k_{\neq}}\one\{v_1=x,v_k=y\}\prod_{i=2}^k\one\{v_i\leftrightarrow v_{i-1}\}\bigg]\\
		&\leq\sum_{(v_1,\dots,v_k)\in[n]^k_{\neq}}\one\{v_1=x,v_k=y\}\prod_{i=2}^k\E_\cW[E_n\{v_i,v_{i-1}\}]\\
		&=\sum_{(v_1,\dots,v_k)\in[n]^k_{\neq}}\one\{v_1=x,v_k=y\}\prod_{i=2}^k\frac{W_{v_i}W_{v_{i-1}}}{L_n}\\
		&\leq\frac{W_xW_y}{L_n}\bigg(\one\{k=2\}+\sum_{v_2,\dots,v_{k-1}=1}^n\prod_{i=2}^{k-1}\frac{W_{v_i}^2}{L_n}\bigg)=\frac{W_xW_y}{L_n}\bigg(\frac{L_n^{(2)}}{L_n}\bigg)^{k-2},
	\end{align*}
	which is the first inequality.
\end{proof}
The next lemma essentially says that vertices with large weights are typically not connected or, equivalently, every component has at most one vertex with large weight. 
\begin{lemma}\label{Lem:prob_of_sets_goes_to_zero}
	For $a>0$ and $n\in\N$ define the event
	\begin{align*}
		\cA_n&=\{\exists x,y\in[n] \colon x\neq y, x\in\cC_n(y), W_{x}\geq W_{y}\geq aq(n)\}^c.
	\end{align*}	
	Under assumption $\mathrm{(W)}$ we have $\p(\cA_n)\to 1$ as $n\to\infty$.
\end{lemma}	
\begin{proof}
	For all $\ell,n\in\N$ we introduce the event
	\begin{align*}
		\cE_{n,\ell}&= \left\{\sum_{y=1}^n\mathbf{1}\{W_y>aq(n)\}\leq\ell\right\}.
	\end{align*}
	Note that $\cE_{n,\ell}$ is $\cW$-measurable. It holds that
	\begin{align}
		\p(\cA_n^c)\leq\p(\cE_{n,\ell}^c)+\p(\cA_n^c\cap \cE_{n,\ell}).\label{rolf}
	\end{align}
	Here the Markov inequality yields
	\begin{align*}
		\p(\cE_{n,\ell}^c)&=\p\left(\sum_{y=1}^n\mathbf{1}\{W_y>aq(n)\}>\ell\right)\leq\frac{\E\left[\sum_{y=1}^n\mathbf{1}\{W_y>aq(n)\}\right]}{\ell}\\
		&=\frac{n}{\ell}\p\left(W>aq(n)\right).
	\end{align*}
	By Theorem 3.6 and Remark 3.3(a) in \cite{Resnick2007} one has
	\begin{align*}
		\lim_{n\to\infty}	\frac{n}{\ell}\p\left(W>aq(n)\right)=\lim_{n\to\infty} \frac{n}{\ell}\cdot\frac{a^{-\beta}}{n}=\frac{a^{-\beta}}{\ell}.
	\end{align*} Now we address the second summand in \eqref{rolf} and show that it converges to zero for any fixed $\ell\in\N$. 
	For $x,y$ as demanded in $\cA_n$ there must exist a path consisting of at least two vertices which connects $x$ and $y$.  Summing over the possible path lengths, we obtain via Lemma \ref{Lem:paths}
	\begin{align*}
		&\E_\mathcal{W}[\one_{\cE_{n,\ell}}\one_{\cA_n^c}]\leq \E_\cW\bigg[\one_{\cE_{n,\ell}}\sum_{x,y=1}^n\one\{x\neq y\}\sum_{k=2}^n|\cP^{(n)}_k(x,y,\emptyset)|\one\{W_x\geq W_y\geq aq(n)\}\bigg]\\
		&\leq \mathbf{1}_{\cE_{n,\ell}}\sum_{x,y=1}^n\one\{W_x\geq W_y\geq aq(n)\}\sum_{k=2}^n\frac{W_xW_y}{L_n}\bigg(\frac{L_n^{(2)}}{L_n}\bigg)^{k-2}\\
		&\leq \mathbf{1}\left\{\sum_{y=1}^n\mathbf{1}\{W_y>aq(n)\}\leq \ell\right\}\sum_{x=1}^n\frac{W_{x}^2\mathbf{1}\{W_{x}>aq(n)\}}{L_n}\times \sum_{y=1}^n\one\{W_{y}> aq(n)\}\sum_{k=2}^n\bigg(\frac{L_n^{(2)}}{L_n}\bigg)^{k-2}\\
		&\leq \sum_{x=1}^n\frac{W_{x}^2\mathbf{1}\{W_{x}>aq(n)\}}{L_n}\ell\, S_{n,\cW}.
	\end{align*}
	The first sum above converges almost surely to zero due to the finite second moment of $W$ and the strong law of large numbers since $q(n)\to\infty$ as $n\to\infty$. By \eqref{conv:Sw}, $S_{n,\cW}$ converges almost surely to a constant so that the whole term converges almost surely to zero. For the conditional expectation we have the upper bound $1$ so that the dominated convergence theorem implies $\p(\cA_n^c\cap\cE_{n,\ell})\to0$ as $n\to\infty$. From \eqref{rolf} we conclude $\limsup_{n\rightarrow\infty}\p(\cA_n^c)\leq a^{-\beta}\cdot \ell^{-1}$ for all $\ell\in\N$. Letting $\ell\to\infty$, we get $\lim_{n\to\infty}\p(\cA_n^c)=0$.	
\end{proof}
We introduce a slightly different variant of assumption (A2) here.\\

(A2')\quad Let $v\in[n]$ and $x\in\cC_n(v)$. Then $v\notin \cX_n(v)$ and checking whether $x\in\cX_n(v)$  only  \hspace*{49pt} depends on all paths that start in $v$ and contain $x$. \\

The difference is that we do not allow $v$ itself to lie in $\cX_n(v)$.
\begin{lemma}\label{Lem:variance_bound_Tnv}
	Assume {\rm{(W), (A2')}} and define
	\begin{align*}
		\tilde{S}_{n,\cW}&=\sum_{k=0}^n(k+3)^3\bigg(\frac{L_n^{(2)}}{L_n}\bigg)^k,	\	X_n=4\tilde{S}_{n,\cW}^2\bigg(1+\frac{W_{(n)}^2}{L_n}+\tilde{S}_{n,\cW}^2\bigg)\quad \text{and}\quad Y_n=4\frac{n}{L_n}\tilde{S}_{n,\cW}^5
	\end{align*} 
	for $n\in\N$. Then $X_n$ and $Y_n$ converge in probability to positive constants as $n\to\infty$ and for all $n\in\N$ and $v\in[n]$,
	\begin{align*}
		\Var_\mathcal{W}\left(S_n(v)\right)\leq W_{v}X_n+W_{v}\frac{L_n^{(3)}}{n}Y_n.
	\end{align*} 
\end{lemma}
\begin{proof}
	The convergence of $X_n$ and $Y_n$ follows from the convergence of $\tilde{S}_{n,\cW}$ in \eqref{conv:rel_Sw} and of $W_{(n)}^2/L_n$ in \eqref{eq:Wn_Ln}. Now we address the actual variance bound. \\
	Conditionally on the weights $\cW$, the quantity $S_n(v)$ depends on the independent random variables $(E_n\{i,j\})_{1\leq i\leq j\leq n}$, which is a Poisson process on the discrete space $\{\{i,j\}\colon 1\leq i\leq j\leq n\}$ with intensity measure $\lambda(\{i,j\})=W_iW_j/L_n$ for $1\leq i\leq j\leq n$. This means that $S_n(v)$ is a Poisson functional which is in particular square-integrable as it is bounded by $n$. We use the Poincar\'e inequality for Poisson functionals, see e.g.\ Theorem 18.7 in \cite{LastPenrose}, to derive
	\begin{align}
		\Var_\mathcal{W}\left(S_n(v)\right)\leq \sum_{1\leq i\leq j\leq n}\E_\cW\left[(D_{\{i,j\}}S_n(v))^2\right]\frac{W_iW_j}{L_n},\label{eq:poincare}
	\end{align}
	where $D_{\{i,j\}}$ denotes the difference operator
	\begin{align*}
		D_{\{i,j\}}S_n(v)=S_{n,i,j}(v)-S_n(v)
	\end{align*}
	and $S_{n,i,j}(v)$ is the number of elements in $\cX_n(v)$ after increasing the number of edges between the vertices $i$ and $j$ by $1$. By assumption (A2'), $v\notin\cX_n(v)$ and for a vertex $x\in\cC_n(v)$ with $x\neq v$ the property $x\in\cX_n(v)$ only depends on all paths starting in $v$ and containing the vertex $x$. For $i=j$ we conclude that $D_{\{i,i\}}S_n(v)=0$ as a path is self-avoiding by definition and therefore not allowed to contain loops. Therefore, we do not alter any paths by adding a loop and hence do not change $S_n(v)$. For $i\neq j$ we obtain
	\begin{align*}
		|D_{\{i,j\}}S_n(v)|\leq |\{x\in[n]\setminus\{v\}\colon &\text{There is a path starting in }v\text{ that contains }x\\
		&\text{and the edge } \{i,j\}\}|,
	\end{align*}
	where the existence of the path is to be checked after increasing the number of edges between $i$ and $j$ by 1. We obtain four different scenarios before the addition of one edge between $i$ and $j$ (independent of the number of edges between $i$ and $j$ before the addition). Figure \ref{pic:P1P2P3P4} contains pictures of the different cases. We write $v\dots x$ for a path that starts in $v$ and ends in $x$.
	\begin{figure}[htpb]
		
		\begin{minipage}{0.2\textwidth}
			\begin{tikzpicture}
				\node[shape=circle,draw=black] (v) at (0,2) {$v$};
				\node[shape=circle,draw=black] (x) at (2,2) {$x$};
				\node[shape=circle,draw=black] (i) at (0,0) {$i$};
				\node[shape=circle,draw=black] (j) at (2,0) {$j$};
				
				\path [dashed] (v) edge (x);
				\path [dashed] (x) edge (i);
				\path [-](i) edge (j);
			\end{tikzpicture}
			\caption*{$\cM_1(i,j)$}
		\end{minipage}
		\hfil
		\begin{minipage}{0.2\textwidth}
			\begin{tikzpicture}
				\node[shape=circle,draw=black] (v) at (0,2) {$v$};
				\node[shape=circle,draw=black] (x) at (2,2) {$x$};
				\node[shape=circle,draw=black] (i) at (0,0) {$i$};
				\node[shape=circle,draw=black] (j) at (2,0) {$j$};
				
				\path [dashed] (v) edge (x);
				\path [dashed] (x) edge (j);
				\path [-](i) edge (j);
			\end{tikzpicture}
			\caption*{$\cM_2(i,j)$}
		\end{minipage}
		\hfil
		\begin{minipage}{0.2\textwidth}
			\begin{tikzpicture}
				\node[shape=circle,draw=black] (v) at (0,2) {$v$};
				\node[shape=circle,draw=black] (x) at (2,2) {$x$};
				\node[shape=circle,draw=black] (i) at (0,0) {$i$};
				\node[shape=circle,draw=black] (j) at (2,0) {$j$};
				
				\path [dashed] (v) edge (i);
				\path [dashed] (x) edge (j);
				\path [-](i) edge (j);
			\end{tikzpicture}
			\caption*{$\cM_3(i,j)$}
		\end{minipage}
		\hfil
		\begin{minipage}{0.2\textwidth}
			\begin{tikzpicture}
				\node[shape=circle,draw=black] (v) at (0,2) {$v$};
				\node[shape=circle,draw=black] (x) at (2,2) {$x$};
				\node[shape=circle,draw=black] (i) at (0,0) {$i$};
				\node[shape=circle,draw=black] (j) at (2,0) {$j$};
				
				\path [dashed] (v) edge (j);
				\path [dashed] (x) edge (i);
				\path [-](i) edge (j);
			\end{tikzpicture}
			\caption*{$\cM_4(i,j)$}
		\end{minipage}
		\caption{Visualisation of $x\in\cM_\ell(i,j)$ for $\ell\in\{1,2,3,4\}$.}
		\label{pic:P1P2P3P4}
	\end{figure}	
	\begin{align*}
		\cM_1(i,j)&=\{x\in[n]\setminus\{v\}\colon \text{There is a path $v\dots i$ that includes $x$ but not $j$.}\}\\
		\cM_2(i,j)&=\{x\in[n]\setminus\{v\}\colon \text{There is a path $v\dots j$ that includes $x$ but not $i$.}\}\\
		\cM_3(i,j)&=\{x\in[n]\setminus\{v\}\colon \text{There are two disjoint paths $v\dots i$ and $j\dots x$.}\}\\
		\cM_4(i,j)&=\{x\in[n]\setminus\{v\}\colon \text{There are two disjoint paths $v\dots j$ and $i\dots x$.}\}
	\end{align*}
	By disjoint paths we mean that they do not share a single vertex. Also, we allow a path to consist of only a single vertex, e.g.\ in $M_3$ the special case $v=i$ immediately yields the existence of a path from $v$ to $i$.  We obtain 
	\begin{align*}
		|D_{\{i,j\}}S_n(v)|&\leq \sum_{k=1}^4|\cM_k(i,j)|
	\end{align*}
	and with Jensen's inequality 
	\begin{align*}
		(D_{\{i,j\}}S_n(v))^2&\leq 4\sum_{k=1}^4|\cM_k(i,j)|^2=4\sum_{k=1}^4|\cM_k(i,j)^2|.
	\end{align*}
	Observe that $\cM_1(i,j)=\cM_2(j,i)$ as well as $\cM_3(i,j)=\cM_4(j,i)$. Recall that the case $i=j$ has no contribution so that \eqref{eq:poincare} simplifies to
	\begin{align}
		\Var_\cW(S_n(v))\leq 4\sum_{(i,j)\in[n]^2_{\neq}}\E_\cW[|\cM_1(i,j)^2|]\frac{W_iW_j}{L_n}+4\sum_{(i,j)\in[n]^2_{\neq}}\E_\cW[|\cM_3(i,j)^2|]\frac{W_iW_j}{L_n}.\label{eq:Poincare_2}
	\end{align}
	We start with bounding the first sum and consider $\E_\cW[|\cM_1(i,j)^2|]$. Let $(x,y)\in \cM_1(i,j)^2$. This means that there is a first path $P$ from $v$ through $x$ to $i$ and a second path from $v$ through $y$ to $i$. Note that $P$ must contain at least two vertices because $x$ is not allowed to equal $v$. We distinguish two cases, see also Figure \ref{pic:P1} for a visualisation. 
	\begin{enumerate}
		\item[a)] $y$ lies on $P$.
		\item[b)] $y$ does not lie on $P$.
	\end{enumerate}
	\begin{figure}[htpb]
		
		\begin{minipage}{0.4\textwidth}
			\begin{tikzpicture}
				\node[shape=circle,draw=black] (v) at (0,1) {$v$};
				\node[shape=circle,draw=black] (x) at (1,0) {$x$};
				\node[shape=circle,draw=black] (y) at (2,1) {$y$};
				\node[shape=circle,draw=black] (i) at (3,0) {$i$};
				\node[shape=circle,draw=black] (j) at (4,1) {$j$};
				\path [dashed] (v) edge (x);
				\path [dashed] (x) edge (y);
				\path [dashed] (y) edge (i);
				\path [-](i) edge (j);
			\end{tikzpicture}
			\caption*{$\rm{a})$}
		\end{minipage}
		\hfil
		\begin{minipage}{0.4\textwidth}
			\begin{tikzpicture}
				\node[shape=circle,draw=black] (v) at (0,1) {$v$};
				\coordinate (h1) at (1,1) {};
				\node[shape=circle,draw=black] (x) at (2,2) {$x$};
				\node[shape=circle,draw=black] (y) at (2,0) {$y$};
				\coordinate (h2) at (3,1) {};
				\node[shape=circle,draw=black] (i) at (4,1) {$i$};
				\node[shape=circle,draw=black] (j) at (5,1) {$j$};
				\path [dashed] (v) edge (h1);
				\path [dashed] (h1) edge (x);
				\path [dashed] (h1) edge (y);
				\path [dashed] (x) edge (h2);
				\path [dashed] (y) edge (h2);
				\path [dashed] (h2) edge (i);
				\path [-](i) edge (j);
			\end{tikzpicture}
			\caption*{$\rm{b})$}
		\end{minipage}
		\hfill
		\caption{Scenarios for $(x,y)\in \cM_1(i,j)^2$.}
		\label{pic:P1}
	\end{figure}	
	We write $M_{1,\rm{a}}(i,j)$ for the number of all $(x,y)\in \cM_1(i,j)^2$ of type a), similarly $M_{1,\rm{b}}(i,j)$ for \rm{b}), so that $|\cM_1(i,j)^2|\leq M_{1,\rm{a}}(i,j)+ M_{1,\rm{b}}(i,j)$. We have
	\begin{align*}
		M_{1,\rm{a}}(i,j)\leq \sum_{k=2}^n(k-1)^2\cdot |\cP^{(n)}_k(v,i,\{i,j\})|
	\end{align*}
	because $x$ and $y$ need to lie on any path from $v$ to $i$ which does not include the vertices $i,j$ as inner vertices. From Lemma \ref{Lem:paths} we conclude that 
	\begin{align*}
		\E_\cW[M_{1,\rm{a}}(i,j)]\leq \sum_{k=2}^n(k-1)^2\frac{W_vW_i}{L_n}\bigg(\frac{L_n^{(2)}}{L_n}\bigg)^{k-2}\leq W_v\frac{W_i}{L_n}\tilde{S}_{n,\cW}.
	\end{align*}
	In case b) we use a similar argument to obtain
	\begin{align*}
		M_{1,\rm{b}}(i,j)&\leq \sum_{k=2}^n\sum_{(p_1,\dots,p_k)\in\cP^{(n)}_k(v,i,\{i,j\})}(k-1)\sum_{1\leq c<d\leq k}\sum_{\ell=3}^n(\ell-2)\cdot|\cP^{(n)}_\ell(p_c,p_d,\{i,j,p_1,\dots,p_k\})|
	\end{align*}
	because $x$ needs to lie one some path $P$ from $v$ to $i$ whereas $y$ needs to lie on another path $Q$ which splits off $P$ and merges with $P$ again. Since $y$ does not lie on $P$ due to case b), $Q$ contains at least three vertices and $y$ can be neither its starting point nor its endpoint, denoted by $p_c$ and $p_d$ above. Note that the path $Q$ can be chosen in such a way that it only intersects $P$ in $p_c$ and $p_d$. We therefore obtain conditional independence of all paths above and use Lemma \ref{Lem:paths} to obtain
	\begin{align*}
		\E_\cW[M_{1,\rm{b}}(i,j)]&\leq \sum_{k=2}^n(k-1) \frac{W_vW_i}{L_n}\bigg(\frac{L_n^{(2)}}{L_n}\bigg)^{k-2}\sum_{1\leq c<d\leq k}\sum_{\ell=3}^n(\ell-2)\frac{W_{(n)}^2}{L_n}\bigg(\frac{L_n^{(2)}}{L_n}\bigg)^{\ell-2}\\
		&\leq W_v\frac{W_i}{L_n}\frac{W_{(n)}^2}{L_n}\sum_{k=2}^nk^3\bigg(\frac{L_n^{(2)}}{L_n}\bigg)^{k-2}\sum_{\ell=1}^n\ell\bigg(\frac{L_n^{(2)}}{L_n}\bigg)^{\ell}\leq  W_v\frac{W_i}{L_n}\frac{W_{(n)}^2}{L_n}\tilde{S}_{n,\cW}^2.
	\end{align*}
	We conclude
	\begin{align*}
		\E_\cW[|\cM_1(i,j)^2|]\leq W_v\frac{W_i}{L_n}\tilde{S}_{n,\cW}\bigg(1+\frac{W_{(n)}^2}{L_n}\tilde{S}_{n,\cW}\bigg).
	\end{align*}
	Therefore, the first sum in \eqref{eq:Poincare_2} is bounded by
	\begin{align}
		&\sum_{(i,j)\in[n]^2_{\neq}}^nW_v\frac{W_i}{L_n}\tilde{S}_{n,\cW}\bigg(1+\frac{W_{(n)}^2}{L_n}\tilde{S}_{n,\cW}\bigg)\frac{W_iW_j}{L_n}\leq W_v\tilde{S}_{n,\cW}^2\bigg(1+\frac{W_{(n)}^2}{L_n}\tilde{S}_{n,\cW}\bigg),\label{eq:bd1}
	\end{align}
	where we used that $L_n^{(2)}/L_n\leq \tilde{S}_{n,\cW}$.
	
	We proceed with a similar strategy for $\cM_3(i,j)^2$. There are some more cases due to the fact that $v$ itself may be equal to $i$ whereas this was impossible in the case of $\cM_1(i,j)^2$. Let $(x,y)\in \cM_3(i,j)^2$. This means that we can find a path $P$ from $v$ to $i$ via $p_2,\dots,p_{u-1}$ where $u=1$ corresponds to the case $v=i$. We fix such a path $P$. Additionally, there is a path from $j$ to $x$ and another path from $j$ to $y$. We distinguish the following cases, see also Figure \ref{pic:P3}. By definition of $\cM_3(i,j)$ all paths $Q$ and $R$ in the cases below may be chosen disjoint from the path $P$ from $v$ to $i$.
	\begin{figure}[htpb]
		
		\begin{minipage}{0.4\textwidth}
			\begin{tikzpicture}
				\node[shape=circle,draw=black] (v) at (0,1) {$v$};
				\node[shape=circle,draw=black] (i) at (1.5,1) {$i$};
				\node[shape=circle,draw=black] (j) at (3,1) {$j$};
				\path [dashed] (v) edge (i);
				\path [-](i) edge (j);
			\end{tikzpicture}
			\caption*{$a)$}
		\end{minipage}
		\hfil
		\begin{minipage}{0.4\textwidth}
			\begin{tikzpicture}
				\node[shape=circle,draw=black] (v) at (0,1) {$v$};
				\node[shape=circle,draw=black] (i) at (1.5,1) {$i$};
				\node[shape=circle,draw=black] (j) at (3,1) {$j$};
				\node[shape=circle,draw=black] (x) at (4.5,2) {$x$};
				\node[shape=circle,draw=black] (y) at (4.5,0) {$y$};
				\path [dashed] (v) edge (i);
				\path [-](i) edge (j);
				\path [dashed] (j) edge (x);
				\path [dashed] (x) edge (y);
			\end{tikzpicture}
			\caption*{$b)$}
		\end{minipage}\\
		
		\begin{minipage}{0.4\textwidth}
			\begin{tikzpicture}
				\node[shape=circle,draw=black] (v) at (0,1) {$v$};
				\node[shape=circle,draw=black] (i) at (1.5,1) {$i$};
				\node[shape=circle,draw=black] (j) at (3,1) {$j$};
				\node[shape=circle,draw=black] (x) at (4.5,2) {$x$};
				\node[shape=circle,draw=black] (y) at (4.5,0) {$y$};
				\path [dashed] (v) edge (i);
				\path [-](i) edge (j);
				\path [dashed] (j) edge (x);
				\path [dashed] (j) edge (y);
			\end{tikzpicture}
			\caption*{$c)$}
		\end{minipage}
		\hfil
		\begin{minipage}{0.4\textwidth}
			\begin{tikzpicture}
				\node[shape=circle,draw=black] (v) at (0,1) {$v$};
				\node[shape=circle,draw=black] (i) at (1.5,1) {$i$};
				\node[shape=circle,draw=black] (j) at (3,1) {$j$};
				\node[shape=circle,draw=black] (x) at (4.5,2) {$x$};
				\node[shape=circle,draw=black] (y) at (4.5,0) {$y$};
				\coordinate (h) at (4,1);
				\path [dashed] (v) edge (i);
				\path [-](i) edge (j);
				\path [dashed] (j) edge (h);
				\path [dashed] (h) edge (x);
				\path [dashed] (h) edge (y);
			\end{tikzpicture}
			\caption*{$d)$}
		\end{minipage}
		\caption{Different scenarios for $(x,y)\in \cM_3(i,j)^2$.}
		\label{pic:P3}
	\end{figure}	
	\begin{enumerate}
		\item[a)] $x=y=j$
		\item[b)] $|\{x,y,j\}|\geq 2$ and there is a path $Q$ that starts in $j$ and contains $x$ and $y$.
		\item[c)] $|\{x,y,j\}|=3$ and there is a path $Q$ from $j$ to $x$ and a path $R$ from $j$ to $y$ such that $j$ is their only common vertex.
		\item[d)] $|\{x,y,j\}|=3$, there is a path $Q$  from $j$ to $x$ and there is a second path $R$ from one of $Q$'s inner vertices to $y$ that is disjoint from $Q$ apart from its first vertex.
	\end{enumerate}
	For fixed $i,j,P$ we write $M_{3,\rm{z}}(i,j,P)$ for the number of all $(x,y)\in\cM_3(i,j)^2$ considered in case $\rm{z})$ for $\rm{z}\in\{a,b,c,d\}$. We have 
	\begin{align*}
		\one\{v\leftrightarrow p_2\leftrightarrow \hdots \leftrightarrow p_{u-1}\leftrightarrow i\}M_{3,\rm{a}}(i,j,P)=\one\{v\leftrightarrow p_2\leftrightarrow \hdots \leftrightarrow p_{u-1}\leftrightarrow i\}
	\end{align*}
	and by counting all possibilities to place $x$ and $y$ on $Q$ we obtain
	\begin{align*}
		&\one\{v\leftrightarrow p_2\leftrightarrow \hdots \leftrightarrow p_{u-1}\leftrightarrow i\}M_{3,\rm{b}}(i,j)\\
		&\leq \one\{v\leftrightarrow p_2\leftrightarrow \hdots \leftrightarrow p_{u-1}\leftrightarrow i\} \sum_{k=2}^n |\cP^{(n)}_k(j,\{v,i,j,p_2,\dots,p_{u-1}\})|k^2,
	\end{align*}
	so that Lemma \ref{Lem:paths} yields
	\begin{align*}
		&\E_\cW[\one\{v\leftrightarrow p_2\leftrightarrow \hdots \leftrightarrow p_{u-1}\leftrightarrow i\}M_{3,\rm{b}}(i,j)]\\
		&\leq \p_\cW(v\leftrightarrow p_2\leftrightarrow \hdots \leftrightarrow p_{u-1}\leftrightarrow i)W_j\sum_{k=2}^nk^2\bigg(\frac{L_n^{(2)}}{L_n}\bigg)^{k-2}\\
		&\leq\p_\cW(v\leftrightarrow p_2\leftrightarrow \hdots \leftrightarrow p_{u-1}\leftrightarrow i)W_j\tilde{S}_{n,\cW}.
	\end{align*}
	For c) we obtain similarly
	\begin{align*}
		&\one\{v\leftrightarrow p_2\leftrightarrow \hdots \leftrightarrow p_{u-1}\leftrightarrow i\}M_{3,\rm{c}}(i,j)\\
		&\leq \one\{v\leftrightarrow p_2\leftrightarrow \hdots \leftrightarrow p_{u-1}\leftrightarrow i\} \sum_{k=2}^nk\sum_{(q_1,\dots,q_k)\in\cP^{(n)}_k(j,\{v,p_2,\dots,p_{u-1},i\})}\\
		&\quad\times \sum_{\ell=2}^n\ell\cdot|\cP^{(n)}_\ell(j,\{v,p_2,\dots,p_{u-1},i,q_1,\dots,q_k\})|.
	\end{align*}
	We derive due to conditional independence
	\begin{align*}
		&\E_\cW[\one\{v\leftrightarrow p_2\leftrightarrow \hdots \leftrightarrow p_{u-1}\leftrightarrow i\}M_{3,\rm{c}}(i,j)]\\
		&\leq \p_\cW(v\leftrightarrow p_2\leftrightarrow \hdots \leftrightarrow p_{u-1}\leftrightarrow i)W_j^2\bigg(\sum_{k=2}^nk\bigg(\frac{L_n^{(2)}}{L_n}\bigg)^{k-2}\bigg)^2\\
		&\leq\p_\cW(v\leftrightarrow p_2\leftrightarrow \hdots \leftrightarrow p_{u-1}\leftrightarrow i)W_j^2\tilde{S}_{n,\cW}^2.
	\end{align*}
	In case d) we have
	\begin{align*}
		&\one\{v\leftrightarrow p_2\leftrightarrow \hdots \leftrightarrow p_{u-1}\leftrightarrow i\}M_{3,\rm{d}}(i,j)\\
		&\leq \one\{v\leftrightarrow p_2\leftrightarrow \hdots \leftrightarrow p_{u-1}\leftrightarrow i\} \sum_{k=3}^nk\sum_{(q_1,\dots,q_k)\in\cP^{(n)}_k(j,\{v,p_2,\dots,p_{u-1},i\})}\\
		&\quad\times\sum_{\ell=2}^n\sum_{1<t<k}\ell\cdot|\cP^{(n)}_\ell(q_t,\{q_1,\dots,q_k,v,p_2,\dots,p_{u-1},i\})|.
	\end{align*}
	When calculating the conditional expectation, we again use conditional independence. This time, the weight $W_{q_t}$ of the vertex in which the second path originates will appear with a third power. It already has a second power since it is an inner vertex of the first path and as starting point of the second path, we obtain another factor. In total, we get
	\begin{align*}
		&\E_\cW[\one\{v\leftrightarrow p_2\leftrightarrow \hdots \leftrightarrow p_{u-1}\leftrightarrow i\}M_{3,\rm{d}}(i,j)]\\
		&\leq \p_\cW(v\leftrightarrow p_2\leftrightarrow \hdots \leftrightarrow p_{u-1}\leftrightarrow i) W_j\frac{L_n^{(3)}}{L_n}\sum_{k=3}^nk^2\bigg(\frac{L_n^{(2)}}{L_n}\bigg)^{k-3}\sum_{\ell=2}^n\ell\bigg(\frac{L_n^{(2)}}{L_n}\bigg)^{\ell-2}\\
		&\leq \p_\cW(v\leftrightarrow p_2\leftrightarrow \hdots \leftrightarrow p_{u-1}\leftrightarrow i) W_j\frac{L_n^{(3)}}{L_n}\tilde{S}_{n,\cW}^2.
	\end{align*}
	Summing over all different choices for $P$ to connect $v$ and $i$ via $p_2,\dots,p_{u-1}$ yields
	\begin{align*}
		&\E_\cW[|\cM_3(i,j)^2|]\\
		&\leq  \bigg(\one\{v=i\}+\p_\cW(v\leftrightarrow i)+ \sum_{u=3}^n\sum_{(p_2,\dots,p_{u-1})\in([n]\setminus\{v,i\})^{u-2}_{\neq}}\hspace{-40pt}\p_\cW(v\leftrightarrow p_2\leftrightarrow \hdots \leftrightarrow p_{u-1}\leftrightarrow i)\bigg)\\
		&\quad\times \bigg(1+W_j\tilde{S}_{n,\cW}+W_j^2\tilde{S}_{n,\cW}^2+W_j\frac{L_n^{(3)}}{L_n}\tilde{S}_{n,\cW}^2\bigg)\\
		&\leq \bigg(1\{v=i\}+\sum_{k=2}^n\E_\cW[|\cP^{(n)}_{k}(v,i,\emptyset)|]\bigg)\bigg(1+W_j\tilde{S}_{n,\cW}+W_j^2\tilde{S}_{n,\cW}^2+W_j\frac{L_n^{(3)}}{L_n}\tilde{S}_{n,\cW}^2\bigg)\\
		&\leq \bigg(1\{v=i\}+\frac{W_vW_i}{L_n}\tilde{S}_{n,\cW}\bigg)\bigg(1+W_j\tilde{S}_{n,\cW}+W_j^2\tilde{S}_{n,\cW}^2+W_j\frac{L_n^{(3)}}{L_n}\tilde{S}_{n,\cW}^2\bigg),
	\end{align*}
	where the last inequality uses Lemma \ref{Lem:paths}. This bounds the second sum in \eqref{eq:Poincare_2} by
	\begin{align}
		&\sum_{(i,j)\in[n]^2_{\neq}}\bigg(1\{v=i\}+\frac{W_vW_i}{L_n}\tilde{S}_{n,\cW}\bigg)\bigg(1+W_j\tilde{S}_{n,\cW}+W_j^2\tilde{S}_{n,\cW}^2+W_j\frac{L_n^{(3)}}{L_n}\tilde{S}_{n,\cW}^2\bigg) \frac{W_iW_j}{L_n}\nonumber\\
		&\leq W_v\bigg(1+\frac{L_n^{(2)}}{L_n}\tilde{S}_{n,\cW}\bigg)\bigg(1+\frac{L_n^{(2)}}{L_n}\tilde{S}_{n,\cW}+\frac{L_n^{(3)}}{L_n}\tilde{S}_{n,\cW}^2+\frac{L_n^{(2)}L_n^{(3)}}{L_n^2}\tilde{S}_{n,\cW}\bigg)\nonumber\\
		&\leq W_v\tilde{S}_{n,\cW}^2\bigg(\tilde{S}_{n,\cW}^2+\frac{L_n^{(3)}}{L_n}\tilde{S}_{n,\cW}^3\bigg),\label{eq:bd2}
	\end{align}
	where we used $1+L_n^{(2)}\tilde{S}_{n,\cW}/L_n\leq \tilde{S}_{n,\cW}^2$. Combining \eqref{eq:Poincare_2} with \eqref{eq:bd1} and \eqref{eq:bd2}, we obtain
	\begin{align*}
		&\Var_\cW(S_n(v))\leq 4W_v \bigg( \tilde{S}_{n,\cW}^2\bigg(1+\frac{W_{(n)}^2}{L_n}\tilde{S}_{n,\cW}\bigg)+\tilde{S}_{n,\cW}^2\bigg(\tilde{S}_{n,\cW}^2+\frac{L_n^{(3)}}{L_n}\tilde{S}_{n,\cW}^3\bigg)\bigg)\\
		&=4W_v\tilde{S}_{n,\cW}^2\bigg(1+\frac{W_{(n)}^2}{L_n}+\tilde{S}_{n,\cW}^2\bigg)+4W_v \frac{L_n^{(3)}}{L_n}\tilde{S}_{n,\cW}^5=W_vX_n+W_v\frac{L_n^{(3)}}{n}Y_n,
	\end{align*}
	which finishes the proof.	
\end{proof}
The following lemma essentially shows that the conditional $m$-th moment of the number of vertices in a component is bounded by a polynomial of degree $m$ in the largest weight of the component, up to a term which converges almost surely. 
\begin{lemma}\label{Lem:bound_mth_moment_of_Tnv}
	Assume {\rm{(W), (A2')}}, let $m\in\N$ and define for $n\in\N$,
	\begin{align*}
		R_{n,m}= m^m\sum_{t=1}^mt!\left(\sum_{k=0}^n\left(\frac{L_n^{(2)}}{L_n}\right)^k(k+3)^t\right)^t.
	\end{align*} Then $R_{n,m}$ converges almost surely to a constant as $n\to\infty$ and for all $v\in[n]$ it holds that
	\begin{align*}
		\E_\mathcal{W}\left[\one\{v\in V^{\max}_{n}\}S_n(v)^m\right]\leq R_{n,m}\sum_{t=1}^m W_v^t.
	\end{align*}
\end{lemma}
\begin{proof}
	To simplify notation we assume without loss of generality that $v=n$. For $m\in\N$, the almost sure convergence of $R_{n,m}$ as $n\to\infty$ follows from \eqref{conv:rel_Sw}. For the moment bound, recall $v\notin\cX_n(v)$ by (A2') so that
	\begin{align*}	
		&\E_\mathcal{W}\left[\one\{v\in V^{\max}_{n}\}S_n(v)^m\right]\leq \E_\mathcal{W}\bigg[\one\{v\in V^{\max}_{n}\}\bigg(\sum_{x\in[n-1]}\one\{x\in\cC_n(v)\}\bigg)^m\bigg]\\
		&=\sum_{(x_1,\dots,x_m)\in[n-1]^m}\E_\mathcal{W}\bigg[\one\{v\in V^{\max}_{n}\}\prod_{i=1}^m\one\{x_i\in\cC_n(v)\}\bigg]\\
		&\leq \sum_{t=1}^mm^m\sum_{(x_1,\dots,x_t)\in[n-1]_{\neq}^t}\E_\mathcal{W}\bigg[\one\{v\in V^{\max}_{n}\}\prod_{i=1}^t\one\{x_i\in\cC_n(v)\}\bigg],
	\end{align*}
	where the last step accounts for equal entries of a vector $(x_1,\dots,x_m)\in[n]^m$ by the additional factor $m^m$ and instead only sums over the $t\in[m]$ distinct entries of $x$. Hence it suffices to show for all $t\in[m]$ that 
	\begin{align*}
		\sum_{(x_1,\dots,x_t)\in[n-1]^t_{\neq}}\hspace{-9pt}\E_\mathcal{W}\bigg[\one\{v\in V^{\max}_{n}\}\prod_{i=1}^t\one\{x_i\in\cC_n(v)\}\bigg]\leq W_v^tt!\bigg(\sum_{k=0}^n\bigg(\frac{L_n^{(2)}}{L_n}\bigg)^k(k+3)^t\bigg)^t.
	\end{align*}
	We order the vertices $x_1,\dots,x_t$ in such a way that the graph distance of $x_i$ and $v$ is non-decreasing in $i$ which gives us the $t!$ on the right-hand side of the inequality above.
	
	We apply an iterative argument to bound the product on the left-hand side above. If $x_1,\dots,x_t\in\cC_n(v)$, there exists a shortest path $P_1$ starting in $v$ and ending in $x_1$, consisting of $k_1\geq 2$ vertices because $x_1\neq v$. There may be several shortest paths and if so, we always choose the smallest one with respect to the lexicographic order of its vertices. Now consider a shortest path from $v$ to $x_2$ and remove the part leading from $v$ to its last intersection with the already existing path $P_1$. The remaining part, connecting $P_1$ and $x_2$, is called $P_2$ and contains $k_2\geq2$ vertices. Applying this construction iteratively, we connect for $i=2,\dots,t$ a vertex $x_i$ via a path $P_i$ of $k_i\geq 2$ vertices to some vertex $a_i$ of the previously added paths $P_1,\dots,P_{i-1}$. Note that we obtain a tree structure without cycles by choosing the shortest path by lexicographic order (if necessary).

	For $i\in[t]$ and any of these paths $P_i$, write $a_i$ for the starting vertex of $P_i$ and recall that the weight $W_v$ is maximal in $\cC_n(v)$. Note that $a_1=n$. For $i\geq 2$, there are at most $k_1+\hdots+k_{i-1}\leq \prod_{j=1}^{{i-1}}(k_j+1)$ possibilities to choose the point $a_i$ on any of the paths $P_1,\dots,P_{i-1}$ where $P_i$ may be attached to. We write $\cup_{j=1}^{i-1}P_j$ for the set of all vertices of $P_1,\dots,P_{i-1}$. Once one has chosen any such a starting vertex $a_i$, the conditional expectation of the number of such paths $P_i$ is bounded by
	\begin{align*}
		&\E_\cW\bigg[\sum_{k_i=2}^n|\cP^{(n)}_{k_i}(a_i,x_i,\cup_{j=1}^{i-1}P_j)|\cdot \one\{W_{a_i}\leq W_v\}\bigg]\leq W_{v}\frac{W_{x_i}}{L_n}\sum_{k_i=2}^n\bigg(\frac{L_n^{(2)}}{L_n}\bigg)^{k_i-2},
	\end{align*}
	where we used Lemma \ref{Lem:paths} and the fact that the weight of $W_v$ is maximal in its component, so in particular not smaller than the weight $W_{a_i}$. By construction the paths $P_1,\dots,P_t$ do not share an edge and are therefore conditionally independent.

	Combining this bound with the number of possible choices for $a_i$ we obtain
	\begin{align*}
		&\sum_{(x_1,\dots,x_t)\in[n-1]^t_{\neq}}\hspace{-5pt}\E_\mathcal{W}\bigg[\one\{v\in V^{\max}_{n}\}\prod_{i=1}^t\one\{x_i\in\cC_n(v)\}\bigg]\\
		&\leq W_v^tt!\hspace{-5pt}\sum_{(x_1,\dots,x_t)\in[n-1]^t_{\neq}}\prod_{i=1}^t\frac{W_{x_i}}{L_n}\sum_{k_i=2}^n \bigg(\frac{L_n^{(2)}}{L_n}\bigg)^{k_i-2}(k_i+1)^{t-i}\leq W_v^tt!\bigg(\sum_{k=0}^\infty\bigg(\frac{L_n^{(2)}}{L_n}\bigg)^k(k+3)^t\bigg)^t,
	\end{align*}
	which concludes the proof.
\end{proof}
Now we have collected all auxiliary lemmas and proceed with the proof of the main theorem.
\begin{proof}[Proof of Theorem \ref{Thm:conv_to_PP}]
	For $n\in\N$ we compare the point processes
	\begin{align*}
		\Theta_n=\sum_{v=1}^n\delta_{W_vq(n)^{-1}}\quad \text{and}\quad \Xi_n= \sum_{v=1}^n\mathbf{1}\{v\in V^{\max}_{n}\}\delta_{S_n(v)q(n)^{-1}\xi^{-1}}.
	\end{align*}	
	For $n\to\infty$, the convergence of $\Theta_n$, the point process of the rescaled weights, to the Poisson process $\eta_\beta$ in $M_p((0,\infty])$ has been established in Lemma 3.6 in \cite{Matthias_Chinmoy}. We show that $\Xi_n$ behaves asymptotically like $\Theta_n$. To be more precise, due to the proof of Theorem 2.1 in \cite{Matthias_Chinmoy} it suffices to show for all $a>0$ that
	\begin{align}
		\Xi_n((a,\infty])-\Theta_n((a,\infty])\overset{\p}{\longrightarrow}0\quad\text{as}\quad n\to\infty, \label{eq:pp_conv_on_a_infty}
	\end{align}
	in order to conclude our main theorem. 	Since $q(n)\to\infty$ as $n\to\infty$, it does not matter whether we consider $S_n(v)$ or $S_n(v)-1$ in $\Xi_n$. Therefore, we may assume without loss of generality that (A2') is satisfied, i.e.\ that $v\not\in\cX_n(v)$ for all $v\in[n]$. 	 It holds that 
	\begin{align*}
		|\Xi_n((a,\infty])-\Theta_n((a,\infty])|&\leq \sum_{v=1}^n\mathbf{1}\{W_v>aq(n)\}\mathbf{1}\{v\notin V^{\max}_{n}\}\\
		&\quad+\sum_{v=1}^n\mathbf{1}\{W_v>aq(n)\}\mathbf{1}\{v\in V^{\max}_{n},S_n(v)\leq aq(n)\xi\}\\
		&\quad+\sum_{v=1}^n\mathbf{1}\{W_v\leq aq(n)\}\mathbf{1}\{v\in V^{\max}_{n},S_n(v)>aq(n)\xi\}\eqqcolon I_1+I_2+I_3.
	\end{align*}
	We show that $I_1,I_2$ and $I_3$ converge to zero in probability.  For $I_1$ it follows from Lemma \ref{Lem:prob_of_sets_goes_to_zero} that
	\begin{align*}
		\p\left(I_1\neq0\right)= \p\left(\exists x,y \in[n]\colon x\neq y,W_x\geq W_y\geq aq(n),x\in\cC_n(y)\right)=\p(\cA_n^c)\to0,
	\end{align*}	
	as $n\to\infty$. We continue with decomposing $I_2$ and $I_3$. To this end, let $\varepsilon\in(0,a)$. It holds that
	\begin{align*}
		&I_2=\sum_{v=1}^n\mathbf{1}\{v\in V^{\max}_{n},S_n(v)\leq aq(n)\xi<W_v\xi\}\\
		&\leq\sum_{v=1}^n\mathbf{1}\{aq(n)<W_v\leq (a+\varepsilon)q(n)\}+\sum_{v=1}^n\mathbf{1}\{W_v>(a+\varepsilon)q(n),S_n(v)\leq aq(n)\xi\}\eqqcolon I_{2,1}+I_{2,2}.		
	\end{align*}
	For $I_3$ we fix some small positive $\gamma$ satisfying
	\begin{align*}
		0<\gamma<\beta^{-1} \quad \text{and define}\quad \tilde{q}(n)= n^{-\gamma}q(n)\quad \text{for }n\in\N.
	\end{align*}
	Note that $\tilde{q}(n)\leq q(n)$, $\tilde{q}(n)q(n)^{-1}\rightarrow0$ and, by \eqref{q(n):scaling}, $\tilde{q}(n)\to\infty$ as $n\rightarrow\infty$. For $n\in\N$ we have
	\begin{align*}
		I_3&= \sum_{v=1}^n\mathbf{1}\{v\in V^{\max}_{n},W_v\xi\leq aq(n)\xi<S_n(v)\}\\
		&\leq\sum_{v=1}^n\mathbf{1}\{(a-\varepsilon)q(n)<W_v\leq aq(n)\}+\sum_{v=1}^n\mathbf{1}\{a\tilde{q}(n)<W_v\leq (a-\varepsilon)q(n),S_n(v)>aq(n)\xi\}\\
		&\quad+\sum_{v=1}^n\mathbf{1}\{v\in V^{\max}_{n},W_v\leq a\tilde{q}(n),S_n(v)>aq(n)\xi\}\eqqcolon I_{3,1}+I_{3,2}+I_{3,3}.
	\end{align*}

	By Theorem 3.6 and Remark 3.3(a) in \cite{Resnick2007} we get
	\begin{align*}
		&\lim_{n\to\infty}\E\bigg[\sum_{v=1}^n\mathbf{1}\{(a-\varepsilon)q(n)<W_v\leq (a+\varepsilon)q(n)\}\bigg]\\
		&=\lim_{n\to\infty}n\p((a-\varepsilon)q(n)<W\leq (a+\varepsilon) q(n))=(a-\varepsilon)^{-\beta}-(a+\varepsilon)^{-\beta},
	\end{align*}
	so that $\lim_{\varepsilon\rightarrow0}\limsup_{n\rightarrow\infty}\E\left[I_{2,1}\right]=0$ and $\lim_{\varepsilon\rightarrow0}\limsup_{n\rightarrow\infty}\E\left[I_{3,1}\right]=0$.
	
	We easily see that
	\begin{align*}
		I_4= \sum_{v=1}^n\mathbf{1}\left\{a\tilde{q}(n)<W_v,|S_n(v)-W_v\xi|>\varepsilon q(n)\xi\right\}
	\end{align*}
	is an upper bound for $I_{2,2}$ and $I_{3,2}$. It remains to show that $I_{3,3}\overset{\p}{\longrightarrow}0$ and $I_{4}\overset{\p}{\longrightarrow}$ as $n\to\infty$. Note that both these random variables take values in $\N_0$. For any sequence of $\N_0$-valued random variables $(\zeta_n)_{n\in\N}$ and any sequence of events $(Q_n)_{n\in\N}$ with $\p(Q_n)\to 1$ as $n\to\infty$ it holds that 
	\begin{align}
		\p(\zeta_n\neq 0)&=\p(\one_{Q_n^c}\zeta_n\neq0)+\p(\one_{Q_n}\zeta_n\neq 0)\leq \p(Q_n^c)+\E\left[\one\{\one_{Q_n}\zeta_n\neq 0\}\right]\nonumber\\
		&=\p(Q_n^c)+\E\left[\min\left(1,\one_{Q_n}\zeta_n\right)\right]=\p(Q_n^c)+\E\left[\E_\mathcal{W}\left[\min\left(1,\one_{Q_n}\zeta_n\right)\right]\right]\nonumber\\
		&\leq \p(Q_n^c)+ \E\left[\min\left(1,\E_\mathcal{W}\left[\one_{Q_n}\zeta_n\right]\right)\right],\label{eq:cond_exp_conv_to_zero}
	\end{align}
	where we used Jensen's inequality in the last step.	Thus, for $\zeta_n\overset{\p}{\longrightarrow}0$ as $n\to\infty$ it suffices to show $\E_\mathcal{W}[\one_{Q_n}\zeta_n]\overset{\p}{\longrightarrow}0$ as $n\to\infty$.

	We start with the summand $I_{3,3}$. As discussed above it suffices to show that $\E_\cW[I_{3,3}]\overset{\p}{\longrightarrow}0$ as $n\to\infty$. For $m\in\N$ we compute
	\begin{align*}
		\E_\mathcal{W}[I_{3,3}]&=\E_\mathcal{W}\bigg[\sum_{v=1}^n\mathbf{1}\{v\in V^{\max}_{n},W_v\leq a\tilde{q}(n),S_n(v)>aq(n)\xi\}\bigg]\\
		&=\sum_{v=1}^n\one\{W_v\leq a\tilde{q}(n)\}\p_\mathcal{W}\bigg(\one\{v\in V^{\max}_{n}\}\cdot S_n(v)>aq(n)\xi\bigg)\\
		&\leq \sum_{v=1}^n\one\{W_v\leq a\tilde{q}(n)\}\frac{\E_\mathcal{W}\left[\one\{v\in V^{\max}_{n}\}\cdot S_n(v)^m\right]}{\left(aq(n)\xi\right)^m},
	\end{align*}
	where the last inequality follows from the Markov inequality. Let $n$ be large enough such that $a\tilde{q}(n)>1$. We can bound the conditional expectation by applying Lemma \ref{Lem:bound_mth_moment_of_Tnv} which leads to
	\begin{align*}
		\E_\mathcal{W}[I_{3,3}]&\leq\left(\frac{1}{aq(n)\xi}\right)^m\sum_{v=1}^n\one\{W_v\leq a\tilde{q}(n)\}R_{n,m}\sum_{t=1}^mW_v^t\\
		&\leq R_{n,m}\left(\frac{1}{aq(n)\xi}\right)^m\sum_{v=1}^nm(a\tilde{q}(n))^m=m\xi^{-m}R_{n,m}\cdot n^{1-m\gamma},
	\end{align*}
	where the last step uses $\tilde{q}(n)=n^{-\gamma}q(n)$. Since $R_{n,m}$ converges almost surely to a constant by Lemma \ref{Lem:bound_mth_moment_of_Tnv}, choosing $m>\gamma^{-1}$ ensures that $\E_\cW[I_{3,3}]$ converges almost surely to zero.

	In order to deal with $I_{4}$ we define  the event
	\begin{align*}
		\cG_{n,\varepsilon}&= \left\{\underset{v=1,\dots,n}{\sup}\left|\E_\cW[S_n(v)]-W_{v}\xi\right|\leq\frac{\varepsilon q(n)\xi}{2}\right\}
	\end{align*}
	which satisfies $\p(\cG_{n,\varepsilon})\to1$ as $n\to\infty$ due to (A1). By the discussion after \eqref{eq:cond_exp_conv_to_zero} we are left to show that $\E_\cW[\one_{ \cG_{n,\varepsilon}}I_4]\overset{\p}{\longrightarrow}0$ as $n\to\infty$. We use the $\cW$-measurability of $\cG_{n,\varepsilon}$ to compute  
	\begin{align*}
		&\E_\cW[\one_{ \cG_{n,\varepsilon}}I_4]=\mathbf{1}_{\cG_{n,\varepsilon}}\sum_{v=1}^n\p_\mathcal{W}\left(W_v>a\tilde{q}(n),\left|S_n(v)-W_v\xi\right|>\varepsilon q(n)\xi\right)\\
		&\leq\sum_{v=1}^n\one\{W_v>a\tilde{q}(n)\}\p_\mathcal{W}\left(\left|\E_\cW[S_n(v)]-W_{v}\xi\right|\leq\frac{\varepsilon q(n)\xi}{2},\left|S_n(v)-W_v\xi\right|>\varepsilon q(n)\xi\right)\\
		&\leq \sum_{v=1}^n\one\{W_v>a\tilde{q}(n)\}\p_\mathcal{W}\left(\left|S_n(v)-\E_\mathcal{W}[S_n(v)]\right|>\frac{\varepsilon q(n)\xi}{2}\right).
	\end{align*}
	We use the Chebyshev inequality and Lemma \ref{Lem:variance_bound_Tnv} to bound this further by 
	\begin{align*}
		&\E_\cW[\one_{ \cG_{n,\varepsilon}}I_{4}]\leq \sum_{v=1}^n\one\{W_v>a\tilde{q}(n)\}\frac{\Var_\mathcal{W}\left(S_n(v)\right)}{\left(\frac{\varepsilon q(n)\xi}{2}\right)^2}\\
		&\leq \sum_{v=1}^n\one\{W_v>a\tilde{q}(n)\}\frac{W_vX_n+W_vn^{-1}L_n^{(3)}Y_n}{\left(\frac{\varepsilon q(n)\xi}{2}\right)^2}\\
		&=\frac{4}{(\varepsilon\xi)^2}X_n\sum_{v=1}^n\one\{W_v>a\tilde{q}(n)\}\frac{W_v}{q(n)^2}+\frac{4}{(\varepsilon\xi)^2}Y_n\sum_{v=1}^n\one\{W_v>a\tilde{q}(n)\}\frac{W_vL_n^{(3)}}{nq(n)^2}\\
		&\eqqcolon \frac{4}{(\varepsilon\xi)^2}(X_nI_X+Y_nI_Y).
	\end{align*}
	From Lemma \ref{Lem:variance_bound_Tnv} we know that $X_n$ and $Y_n$ converge in probability to positive constants as $n\to\infty$. By Slutsky's theorem it suffices to show that $I_X+I_Y\overset{\p}{\longrightarrow}0$ as $n\to\infty$.
	
	If $\beta\in(2,3]$, we choose some $\tau\in(0,\beta)$. We define
	\[
	p = \left\{\begin{array}{ll}
		1, &\text{for } \beta>3,\\
		\frac{\beta-\tau}{3}, &\text{for } \beta\in(2,3].
	\end{array}\right.
	\]
	Since then $\E[W^{3p}]<\infty$, it follows from the Marcinkiewicz-Zygmund strong law of large numbers, see e.g.\ Theorem 5.23 in \cite{Kallenberg2021}, that
	\[
	\frac{L_n^{(3)}}{n^{1/p}} \overset{a.s.}{\longrightarrow} \left\{\begin{array}{ll}
		\E[W^3], &\text{for } \beta>3,\\
		0, &\text{for } \beta\in(2,3],
	\end{array}\right. \quad \text{as}\quad n\to\infty.
	\]
	Together with Slutsky's theorem, we see that
	\begin{align*}
		I_X+I_Y\overset{\p}{\longrightarrow}0\quad\text{as}\quad n\to\infty
	\end{align*}
	if
	\begin{align*}
		\frac{1+n^{1/p-1}}{q(n)^2}\sum_{v=1}^n\one\{W_v>a\tilde{q}(n)\}W_v\overset{\p}{\longrightarrow}0\quad\text{as}\quad n\to\infty.
	\end{align*}
	In the following we prove this by showing that its expectation
	\begin{align*}
		h(n)=\frac{n+n^{1/p}}{q(n)^2}\E[\one\{W>a\tilde{q}(n)\}W] 
	\end{align*}
	vanishes for $n\to\infty$.
	For a function $g\colon(0,\infty)\to[0,\infty)$ and $\rho\in\R$ we write $g\in\rm{RV}_\rho$ if $g$ is regularly varying with index $\rho$ at infinity, i.e.\ if there exists a slowly varying function $L\colon(0,\infty)\to[0,\infty)$ such that $g(t)=L(t)t^\rho$ for all $t\in(0,\infty)$.
	
	From Lemma 3.3 in \cite{Matthias_Chinmoy} we know that $q\in\rm{RV}_{1/\beta}$, which implies $\tilde{q}\in\rm{RV}_{1/\beta-\gamma}$. For $u>0$ we have
	\begin{align*}
		\E[\one\{W>u\}W]=\int_0^\infty \p(\one\{W>u\}W>t)dt= u\p(W>u)+\int_{u}^\infty \p(W>t)dt.
	\end{align*}
	Since, by assumption (W), $u\mapsto\p(W>u)$ belongs to $\rm{RV}_{-\beta}$, $u\mapsto u\p(W>u)$ is from $\rm{RV}_{1-\beta}$ and, by Karamata's theorem, see Theorem 2.1 in \cite{Resnick2007}, we derive that $u\mapsto \int_u^\infty\p(W>t)dt$ is also from $\rm{RV}_{1-\beta}$. Together with $\tilde{q}\in \rm{RV}_{1/\beta-\gamma}$ and Proposition 2.6 in \cite{Resnick2007}, we obtain that
	\begin{align*}
		n\mapsto\E[\one\{W>a\tilde{q}(n)\}W]
	\end{align*}
	belongs to $\rm{RV}_{(1-\beta)(1/\beta-\gamma)}$. As $1/p\geq 1$ and $q\in\rm{RV}_{1/\beta}$, we have $h\in{\rm{RV}}_{(1-\beta)(1/\beta-\gamma)+1/p-2/\beta}$. 
	
	If $\beta>3$,
	\begin{align*}
		(1-\beta)(1/\beta-\gamma)+1/p-2/\beta=-\frac{1+\beta}{\beta}+1+\gamma(\beta-1),
	\end{align*}
	while for $\beta\in(2,3]$,
	\begin{align*}
		(1-\beta)(1/\beta-\gamma)+1/p-2/\beta=-\frac{1+\beta}{\beta}+\frac{3}{\beta-\tau}+\gamma(\beta-1).
	\end{align*}
	Now we can choose $\gamma$ and $\tau$ sufficiently small so that the expressions become negative for both cases. Then, we have $h(n)\to0$ as $n\to\infty$, which implies
	\begin{align*}
		\E_\cW[\one_{ \cG_{n,\varepsilon}}I_4]\overset{\p}{\longrightarrow}0 \quad \text{and} \quad I_4\overset{\p}{\longrightarrow}0 \quad \text{as}\quad n\to\infty.
	\end{align*}
	This concludes the proof. 
\end{proof}

\subsection{Proof of Theorem \ref{Lem:assump_for_different_vertices}}

It is clear that all four classes of vertices in Theorem \ref{Lem:assump_for_different_vertices} satisfy condition (A2). The idea to prove (A1) is to explore the component $\cC_n(v)$ starting from the vertex $v$.  We start with a lemma for the connection probabilities.

\begin{lemma}\label{lem:prob_bound_connection}
	For $x,y\in[n]$,
	\begin{align*}
		\frac{W_xW_y}{L_n}\bigg(1-\min\bigg(1,\frac{W_{(n)}^2}{L_n}\bigg)\bigg)\leq \p_\cW(x\leftrightarrow y)\leq \frac{W_xW_y}{L_n}.
	\end{align*}
\end{lemma}
\begin{proof}
	By a Taylor expansion we have 
	\begin{align*}
		\p_\cW(x\leftrightarrow y)=1-\exp\bigg(-\frac{W_xW_y}{L_n}\bigg)=\frac{W_xW_y}{L_n}-\frac{e^{-z}}{2}\bigg(\frac{W_xW_y}{L_n}\bigg)^2
	\end{align*}
	for some $z>0$. With
	\begin{align*}
		0\leq \frac{e^{-z}}{2}\bigg(\frac{W_xW_y}{L_n}\bigg)^2\leq \frac{W_xW_y}{L_n}\frac{W_{(n)}^2}{L_n}
	\end{align*}
	and $\p_\cW(x\leftrightarrow y)\geq 0$, the claim follows.
\end{proof}
For a vertex $v_0\in[n]$ we recall from \eqref{eq:T_nv}  the random variable  
\begin{align*}
	T_n(v_0)=\sum_{k=1}^n\sum_{(v_1,\dots,v_k)\in([n]\setminus\{v_0\})^k_{\neq}}\prod_{i=1}^k\one\{v_i\leftrightarrow v_{i-1}\}\one\{v_k\in\cX_n(v_0)\},
\end{align*}
which checks for all possible paths $v_0\dots v_k$ whether they exist or not and if the endpoint belongs to $\cX_n(v_0)$. If the component of $v_0$ is a tree, i.e.\ it has no cycles, then there is a unique path between any two of its vertices so that $T_n(v_0)$ equals $S_n(v_0)$ -- if $v_0\notin\cX_n(v_0)$. However, if there is a cycle, it is possible that $T_n(v_0)$ counts some vertices more than once. We define the event
\begin{align*}
	\cB_{n}(v_0)&=\left\{\exists k\geq 3,(s_1,\dots,s_k)\in\cC_n(v_0)^{k}_{\neq}\colon s_1\leftrightarrow\hdots \leftrightarrow s_k\leftrightarrow s_1\right\}^c,
\end{align*}
which means that there are no cycles in $v_0$'s component, and the random variable
\begin{align*}
	\overline{T}_n(v_0)= 	\sum_{k=1}^n\sum_{(v_1,\dots,v_k)\in([n]\setminus\{v_0\})^k_{\neq}}\prod_{i=1}^k\one\{v_i\leftrightarrow v_{i-1}\},
\end{align*}
which is an upper bound for $T_n(v_0)$. The following lemma essentially shows that cycles are unlikely.
\begin{lemma}\label{Lem:indicator_no_cycles_hardly_changes_Tnv}
	Assume $\mathrm{(W)}$. For $n\in\N$ and $v_0\in[n]$ we have
	\begin{align*}
		\E_\mathcal{W}\left[\one_{\cB_{n}(v_0)^c}\overline{T}_{n}(v_0)\right]\leq W_{v_0}\frac{W_{(n)}^2}{L_n}U_n,
		\text{ where } U_n=2\bigg(\sum_{k=0}^\infty (k+2)^2\bigg(\frac{L_n^{(2)}}{L_n}\bigg)^k\bigg)^2.
	\end{align*}
\end{lemma}
\begin{proof}
	Assume without loss of generality that $v_0=n$. Then
	\begin{align*}
		\one_{\cB_n(v_0)^c}\overline{T}_{n}(v_0)= \one_{\cB_n(v_0)^c}\sum_{k=1}^n\sum_{(v_1,\dots,v_k)\in [n-1]^{k}_{\neq}}\prod_{i=1}^k\one\{v_i\leftrightarrow v_{i-1}\}.
	\end{align*}
	For $\one_{\cB_n(v_0)^c}\overline{T}_{n}(v_0)$ to be non-zero, there must be a cycle somewhere in the component of $v_0$ by the definition of $\cB_n(v_0)$. Consider the path $v_0\dots v_k$ currently counted in $\overline{T}_{n}(v_0)$. For the existence of a cycle in the component of $v_0$ we obtain two possible cases:
	\begin{enumerate}
		\item Two vertices of the path, say $v_p$ and $v_q$ for $p< q$, can be connected by a different path $s_1\dots s_\ell$ with $s_1=v_p$ and $s_\ell=v_q$, resulting in a cycle. We may assume that this new path has no further intersections with the originally considered path by shortening it if needed. 
		\item There is a path $s_1\dots s_{\ell}$ with $\ell\geq 3$ starting in $s_1=v_p$ for some $0\leq p\leq k$ such that $s_\ell \leftrightarrow s_q$ for some $q<\ell-1$. We may assume that $s_1\dots s_\ell$ has no intersections other than $s_1$ with $v_0\dots v_k$ by shortening $s_1\dots s_\ell$ or ending up in the first case.
	\end{enumerate}
	The two cases above yield the following upper bound for $\one_{\cB_{n}(v_0)^c}\overline{T}_n(v_0)$, where we use the notation from Lemma \ref{Lem:paths},
	\begin{align*}
		&\sum_{k=1}^n\sum_{(v_1,\dots,v_k)\in [n-1]^{k}_{\neq}}\prod_{i=1}^k\one \{v_i\leftrightarrow v_{i-1}\}\sum_{0\leq p<q\leq k}\sum_{\ell=2}^n|\cP^{(n)}_\ell(v_p,v_q,\{v_0,\dots,v_k\})|\\
		&+ \sum_{k=1}^n\sum_{(v_1,\dots,v_k)\in [n-1]^{k}_{\neq}}\prod_{i=1}^k\one\{v_i\leftrightarrow v_{i-1}\}\sum_{p=0}^k\sum_{\ell=3}^n\sum_{x\in[n]\setminus\{v_0,\dots,v_k\}}\\
		&\quad\times\sum_{(s_1,\dots,s_\ell)\in\cP^{(n)}_\ell(v_p,x,\{v_0,\dots,v_k\})}\sum_{q=1}^{\ell-2}\one\{x\leftrightarrow s_q\}\eqqcolon I_1+I_2.
	\end{align*} Without any shared edges, we can use conditional independence and apply Lemma \ref{lem:prob_bound_connection} and Lemma \ref{Lem:paths}  to obtain 
	\begin{align*}
		\E_\cW[I_1]&\leq\sum_{k=1}^n\sum_{(v_1,\dots,v_k)\in [n-1]^{k}_{\neq}}\prod_{i=1}^k\frac{W_{v_i}W_{v_{i-1}}}{L_n}\sum_{p,q=0}^k\sum_{\ell=2}^n\bigg(\frac{L_n^{(2)}}{L_n}\bigg)^{\ell-2}\frac{W_{(n)}^2}{L_n}\\
		&\leq W_{v_0}\frac{W_{(n)}^2}{L_n}\sum_{k=1}^n\bigg(\frac{L_n^{(2)}}{L_n}\bigg)^{k-1}(k+1)^2\sum_{\ell=2}^n\bigg(\frac{L_n^{(2)}}{L_n}\bigg)^{\ell-2}\leq W_{v_0}\frac{W_{(n)}^2}{L_n}\frac{U_n}{2}.
	\end{align*}
	For the second summand we use
	\begin{align*}
		\E_\cW\bigg[\sum_{q=1}^{\ell-1}\one\{x\leftrightarrow s_q\}\bigg] \leq \ell \frac{W_xW_{(n)}}{L_n}
	\end{align*}
	so that we can bound $\E_\cW[I_2]$ by 
	\begin{align*}
		&\sum_{k=1}^n\sum_{(v_1,\dots,v_k)\in [n-1]^{k}_{\neq}}\prod_{i=1}^k\frac{W_{v_i}W_{v_{i-1}}}{L_n}\sum_{p=0}^k\sum_{\ell=3}^n\sum_{x\in[n]\setminus\{v_0,\dots,v_k\}}\frac{W_{v_p}W_x}{L_n}\bigg(\frac{L_n^{(2)}}{L_n}\bigg)^{\ell-2}\ell\frac{W_xW_{(n)}}{L_n}\\
		&\leq W_{v_0}\frac{W_{(n)}^2}{L_n}\sum_{k=1}^n\bigg(\frac{L_n^{(2)}}{L_n}\bigg)^{k-1}(k+1)\sum_{\ell=3}^n\bigg(\frac{L_n^{(2)}}{L_n}\bigg)^{\ell-2}\ell\frac{\sum_{x=1}^nW_x^2}{L_n}
		\leq W_{v_0}\frac{W_{(n)}^2}{L_n}\frac{U_n}{2},
	\end{align*}
	where the last inequality uses $\sum_{x=1}^nW_x^2=L_n^{(2)}$ and an index shift. Summing up both bounds yields the claim.
\end{proof}	

\begin{lemma}\label{lemma:1}
	Let $\xi>0$. Assume {\rm{(W)}} and that there exists for all $n\in\N,k\in[n]$ and $(v_0,\dots,v_k)\in[n]^{k+1}_{\neq}$ an event $\cI_n(v_0,\dots,v_k)$ such that
	\begin{align}
		& \cI_n(v_0,\dots,v_k)\text{ is conditionally on }\cW \text{ independent of }\{v_0\leftrightarrow v_1\},\hdots,\{v_{k-1}\leftrightarrow v_k\},\label{eq:lem1_as3}\\ 
		&\one_{B_n(v_0)}\one\{v_0\leftrightarrow\hdots\leftrightarrow v_k,v_k\in\cX_n(v_0)\}=\one_{B_n(v_0)}\one\{v_0\leftrightarrow\hdots\leftrightarrow v_k\}\one_{\cI_n(v_0,\dots,v_k)} \label{eq:lem1_as1}
	\end{align}
	and
	\begin{align}
		&\frac{1}{q(n)}\sup_{v_0\in[n]}\bigg|\sum_{k=1}^n\sum_{(v_1,\dots,v_k)\in([n]\setminus\{v_0\})^k_{\neq}}\prod_{i=1}^k\frac{W_{v_i}W_{v_{i-1}}}{L_n}\p_\cW(\cI_n(v_0,\dots,v_k))-W_{v_0}\xi\bigg|\overset{\p}{\longrightarrow}0,\label{eq:lem1_as2}
	\end{align}
	as $n\to\infty$. Then, assumption {\rm{(A1)}} is satisfied for that choice of $\xi$.
\end{lemma}
\begin{proof}
	We need to show
	\begin{align*}
		\frac{1}{q(n)}\underset{v=1,\dots,n}{\sup}\left|\E_\cW[S_n(v)]-W_{v}\xi\right|\overset{\p}{\longrightarrow} 0\quad \text{as}\quad n\to\infty.
	\end{align*}
	For $v\in[n]$, adding $v$ to $\cX_n(v)$ or removing $v$ from $\cX_n(v)$ changes the value of $S_n(v)$ by one. Due to the triangle inequality and $q(n)\to\infty$ as $n\to\infty$, this does not change whether the statement above is true or false. Therefore, we may assume without loss of generality that $v\notin\cX_n(v)$ for all $v\in[n]$. 	
	For all $v_0\in[n]$ we write 
	\begin{align*}
		\widetilde{T}_n(v_0)=\sum_{k=1}^n\sum_{(v_1,\dots,v_k)\in([n]\setminus\{v_0\})^k_{\neq}}\prod_{i=1}^k\one\{v_i\leftrightarrow v_{i-1}\} \one_{\cI_n(v_0,\dots,v_k)}
	\end{align*}
	and obtain with \eqref{eq:lem1_as1} and $v\notin\cX_n(v)$ for all $v\in[n]$ that
	\begin{align*}
		\one_{\cB_{n}(v)}S_n(v)=\one_{\cB_{n}(v)}T_n(v)=\one_{\cB_{n}(v)}\widetilde{T}_n(v).
	\end{align*}	
	We conclude 
	\begin{align*}
		\E_\cW[S_n(v)]&=\E_\cW[\one_{\cB_{n}(v)}S_n(v)]+\E_\cW[\one_{\cB_{n}(v)^c}S_n(v)]\\
		&=\E_\cW[\one_{\cB_{n}(v)}\widetilde{T}_n(v)]+\E_\cW[\one_{\cB_{n}(v)^c}S_n(v)]\\
		&=\E_\cW[\widetilde{T}_n(v)]-\E_\cW[\one_{\cB_{n}(v)^c}\widetilde{T}_n(v)]+\E_\cW[\one_{\cB_{n}(v)^c}S_n(v)]
	\end{align*}
	so that $S_n(v),\widetilde{T}_n(v)\leq \overline{T}_n(v)$ and Lemma \ref{Lem:indicator_no_cycles_hardly_changes_Tnv} yield
	\begin{align*}
		|\E_\cW[S_n(v)]-\E_\cW[\widetilde{T}_n(v)]|\leq 2\E_\cW[\one_{\cB_{n}(v)^c}\overline{T}_n(v)]\leq 2W_v\frac{W_{(n)}^2}{L_n}U_n\leq 2\frac{W_{(n)}^3}{L_n}U_n.
	\end{align*}
	This in turn provides us with
	\begin{align*}
		\frac{1}{q(n)}\sup_{v\in[n]}|\E_\cW[S_n(v)]-W_v\xi|\leq \frac{1}{q(n)}\sup_{v\in[n]}|\E_\cW[\widetilde{T}_n(v)]-W_v\xi|+\frac{2W_{(n)}^3}{q(n)L_n}U_n.
	\end{align*}
	Here the last summand converges in probability to zero as $n\to\infty$, see \eqref{conv:rel_Sw} for $U_n$ and \eqref{W(n):scaling} as well as \eqref{eq:Wn_Ln} for the other factors. Next, we bound the supremum on the right-hand side. By \eqref{eq:lem1_as3} we have
	\begin{align*}
		\E_\cW[\widetilde{T}_{n}(v_0)]&=\E_\cW\bigg[\sum_{k=1}^n\sum_{(v_1,\dots,v_k)\in([n]\setminus\{v_0\})^k_{\neq}}\prod_{i=1}^k \one\{v_i\leftrightarrow v_{i-1}\}\one_{\cI_n(v_0,\dots,v_k)}\bigg]\\
		&=\sum_{k=1}^n\sum_{(v_1,\dots,v_k)\in([n]\setminus\{v_0\})^k_{\neq}}\prod_{i=1}^k \p_\cW(v_i\leftrightarrow v_{i-1})\p_\cW(\cI_n(v_0,\dots,v_k))
	\end{align*}
	so that
	\begin{align*}
		&\frac{1}{q(n)}\sup_{v_0\in[n]}|\E_\cW[\widetilde{T}_n(v_0)]-W_{v_0}\xi|\\
		&\leq  \frac{1}{q(n)}\sup_{v_0\in[n]}\bigg|\sum_{k=1}^n\sum_{(v_1,\dots,v_k)\in([n]\setminus\{v_0\})^k_{\neq}}\prod_{i=1}^k \frac{W_{v_i}W_{v_{i-1}}}{L_n}\p_\cW(\cI_n(v_0,\dots,v_k))-W_{v_0}\xi\bigg|\\
		&\quad+ \frac{1}{q(n)}\sup_{v_0\in[n]}\sum_{k=1}^n\sum_{(v_1,\dots,v_k)\in([n]\setminus\{v_0\})^k_{\neq}}\bigg|\prod_{i=1}^k \p_\cW(v_i\leftrightarrow v_{i-1})-\prod_{i=1}^k\frac{W_{v_i}W_{v_{i-1}}}{L_n}\bigg|.
	\end{align*}
	By assumption \eqref{eq:lem1_as2}, the first term on the right-hand side converges in probability to zero as $n\to\infty$. For the second summand we obtain from Lemma \ref{lem:prob_bound_connection} for all $v_0,\dots,v_k\in[n]$,
	\begin{align*}
		&\bigg|\prod_{i=1}^k \p_\cW(v_i\leftrightarrow v_{i-1})-\prod_{i=1}^k\frac{W_{v_i}W_{v_{i-1}}}{L_n}\bigg|\leq \prod_{i=1}^k\frac{W_{v_i}W_{v_{i-1}}}{L_n}\bigg(1-\bigg(1-\min\bigg(1,\frac{W_{(n)}^2}{L_n}\bigg)\bigg)^k\bigg)\\
		&=W_{v_0}\prod_{i=1}^{k-1}\frac{W_{v_i}^2}{L_n}\frac{W_{v_k}}{L_n}\bigg(1-\bigg(1-\min\bigg(1,\frac{W_{(n)}^2}{L_n}\bigg)\bigg)^k\bigg)\leq W_{v_0}\prod_{i=1}^{k-1}\frac{W_{v_i}^2}{L_n}\frac{W_{v_k}}{L_n}k\frac{W_{(n)}^2}{L_n},
	\end{align*}
	where the last inequality uses that $1-y^k\leq k(1-y)$ for all $y\in[0,1]$ by the mean value theorem. We obtain 
	\begin{align*}
		&	\frac{1}{q(n)}\sup_{v_0\in[n]}\sum_{k=1}^n\sum_{(v_1,\dots,v_k)\in([n]\setminus\{v_0\})^k_{\neq}}\bigg|\prod_{i=1}^k \p_\cW(v_i\leftrightarrow v_{i-1})-\prod_{i=1}^k\frac{W_{v_i}W_{v_{i-1}}}{L_n}\bigg|\\
		&\leq \frac{1}{q(n)}\sup_{v_0\in[n]}\sum_{k=1}^n\sum_{(v_1,\dots,v_k)\in[n]^k} W_{v_0}\prod_{i=1}^{k-1}\frac{W_{v_i}^2}{L_n}\frac{W_{v_k}}{L_n}k\frac{W_{(n)}^2}{L_n}=\frac{W_{(n)}^3}{q(n)L_n}\sum_{k=1}^nk\bigg(\frac{L_n^{(2)}}{L_n}\bigg)^{k-1},
	\end{align*}
	which converges in probability to zero as $n\to\infty$ due to \eqref{W(n):scaling}, \eqref{eq:Wn_Ln} and \eqref{conv:rel_Sw}.
\end{proof}
\begin{lemma}\label{lemma:2}
	Assume {\rm{(W)}}. Suppose that for all $n\in\N,k\in[n]$ and $(v_0,\dots,v_k)\in[n]^{k+1}_{\neq}$ there exists an event $\cI_n(v_0,\dots,v_k)$ such that \eqref{eq:lem1_as3} as well as \eqref{eq:lem1_as1} hold. Moreover, assume that there exist a bounded, measurable function  $g\colon[0,\infty)\times\N\to[0,\infty)$, a polynomial $p$ of finite degree and random variables $(R_n)_{n\in\N}$ such that for all $n\in\N,k\in[n]$ and $(v_0,\dots,v_k)\in[n]^{k+1}_{\neq}$,
	\begin{align}
		&|\p_\cW(\cI_n(v_0,\dots,v_k))-g(W_{v_k},k)|\leq p(k)R_n\label{eq:lem2_bound_3}
	\end{align}
	and
	\begin{align}
		&R_n\overset{\p}{\longrightarrow}0 \quad\text{as}\quad n \to\infty.\label{eq:Rn_small}
	\end{align}
	Then, {\rm{(A1)}} is satisfied with\begin{align*}
		\xi=\sum_{k=1}^\infty \bigg(\frac{\E[W^2]}{\E[W]}\bigg)^{k-1}\frac{\E[Wg(W,k)]}{\E[W]}.
	\end{align*}
\end{lemma}
\begin{proof}
	By Lemma \ref{lemma:1} it suffices to show \eqref{eq:lem1_as2}. We have
	\begin{align*}
		&\frac{1}{q(n)}\sup_{v_0\in[n]}\bigg|\sum_{k=1}^n\sum_{(v_1,\dots,v_k)\in([n]\setminus\{v_0\})^k_{\neq}}\prod_{i=1}^k\frac{W_{v_i}W_{v_{i-1}}}{L_n}\p_\cW(\cI_n(v_0,\dots,v_k))-W_{v_0}\xi\bigg|\\
		&\leq \frac{1}{q(n)}\sup_{v_0\in[n]}\bigg|\sum_{k=1}^n\sum_{(v_1,\dots,v_k)\in([n]\setminus\{v_0\})^k_{\neq}}\prod_{i=1}^k\frac{W_{v_i}W_{v_{i-1}}}{L_n}\bigg(\p_\cW(\cI_n(v_0,\dots,v_k))-g(W_{v_k},k)\bigg)\bigg|\\
		&\quad+\frac{1}{q(n)}\sup_{v_0\in[n]}\bigg|\sum_{k=1}^n\sum_{(v_1,\dots,v_k)\in[n]^k\setminus([n]\setminus\{v_0\})^k_{\neq}}\prod_{i=1}^k\frac{W_{v_i}W_{v_{i-1}}}{L_n}g(W_{v_k},k)\bigg|\\
		&\quad+\frac{1}{q(n)}\sup_{v_0\in[n]}\bigg|\sum_{k=1}^n\sum_{(v_1,\dots,v_k)\in[n]^k}\prod_{i=1}^k\frac{W_{v_i}W_{v_{i-1}}}{L_n}g(W_{v_k},k)-W_{v_0}\xi\bigg|\\
		&\eqqcolon I_1+I_2+I_3.
	\end{align*}
	Due to \eqref{eq:lem2_bound_3} we have
	\begin{align*}
		I_1\leq \frac{W_{(n)}}{q(n)}\sum_{k=1}^n\bigg(\frac{L_n^{(2)}}{L_n}\bigg)^{k-1}p(k)R_n,
	\end{align*}
	which converges to zero in probability as $n\to\infty$, by \eqref{W(n):scaling}, \eqref{conv:rel_Sw} and \eqref{eq:Rn_small}. For $I_2$, we consider the elements $(v_1,\dots,v_k)\in[n]^k\setminus([n]\setminus\{v_0\})^k_{\neq}$ for a fixed $v_0\in[n]$, which means that $(v_1,\dots,v_k)$ has two equal entries or contains $v_0$. There are $\binom{k}{2}$ choices for $i,j\in[k]$ with $v_i=v_j$ for $i<j$. In this case, we may bound one factor by $W_{(n)}^2/L_n$ and omit the summation over $v_i$. On the other hand, it may be that there exists $i\in[k]$ such that $v_i=v_0$. If $i<k$, we have $(k-1)$ choices and can once more omit the summation over $v_i$ and bound the respective factor by $W_{(n)}^2/L_n$. For $i=k$, we obtain a different scenario due to the factor $g(W_{v_k},k)$. We derive
	\begin{align*}
		I_2&\leq \frac{W_{(n)}^3}{L_nq(n)}\sum_{k=2}^n\bigg(\binom{k}{2}+k\bigg)\bigg(\frac{L_n^{(2)}}{L_n}\bigg)^{k-2}\frac{\sum_{i=1}^nW_ig(W_i,k)}{L_n}\\
		&\quad +\sum_{k=1}^n \frac{\sup_{v_0\in[n]}W_{v_0}^2g(W_{v_0},k)}{L_nq(n)}\bigg(\frac{L_n^{(2)}}{L_n}\bigg)^{k-1}.
	\end{align*}
	In the first summand the first  factor vanishes as $n\to\infty$ due to \eqref{W(n):scaling} and \eqref{eq:Wn_Ln}, while the two sums are bounded because of \eqref{conv:rel_Sw} as well as the boundedness of $g$ and the law of large numbers. For the second summand, the first fraction converges in probability to zero as $n\to\infty$ by the boundedness of $g$, \eqref{q(n):scaling} and \eqref{eq:Wn_Ln}, whereas the remaining sum converges almost surely as $n\to\infty$ by \eqref{conv:Sw}.
	
	For $I_3$ we obtain
	\begin{align*}
		I_3&= \frac{W_{(n)}}{q(n)}\bigg|\sum_{k=1}^n\bigg(\frac{L_n^{(2)}}{L_n}\bigg)^{k-1}\frac{\sum_{i=1}^nW_ig(W_i,k)}{L_n}-\xi\bigg|.
	\end{align*}
	From the strong law of large numbers and the boundedness of $g$ we deduce
	\begin{align*}
		\sum_{k=1}^n\bigg(\frac{L_n^{(2)}}{L_n}\bigg)^{k-1}\frac{\sum_{i=1}^nW_ig(W_i,k)}{L_n}\overset{a.s.}{\longrightarrow}\sum_{k=1}^\infty\bigg(\frac{\E[W^2]}{\E[W]}\bigg)^{k-1}\frac{\E[Wg(W,k)]}{\E[W]}=\xi,
	\end{align*}
	as $n\to\infty$, which finishes the proof.
\end{proof}
We show that all our classes of vertices in Theorem \ref{Lem:assump_for_different_vertices} meet the assumptions of Lemma \ref{lemma:2}.
\begin{proof}[Proof that \textbf{all vertices} satisfy (A1)]
	For $n\in\N,k\in[n]$ and $(v_0,\dots,v_k)\in[n]^{k+1}_{\neq}$  we choose $\cI_n(v_0,\dots,v_k)$ as an event with probability one, $g=1$ as well as $p=R_n=0$. In this case it is easy to see that all assumptions of Lemma \ref{lemma:2} are met.
\end{proof}

\begin{proof}[Proof that \textbf{vertices in a fixed distance } $m$ \textbf{ to } $v_0$ satisfy (A1)]
	For $n\in\N,k\in[n]$ as well as  $(v_0,\dots,v_k)\in[n]^{k+1}_{\neq}$ we choose $\cI_n(v_0,\dots,v_k)=\{k=m\}$, $g=\one\{k=m\}$ and $p=R_n=0$. We see that all assumptions of Lemma \ref{lemma:2} are met.
\end{proof}

\begin{proof}[Proof that \textbf{vertices with a fixed degree } $m$ satisfy (A1)]
	We start with the case $m=1$, i.e.\ we consider leaves. We define for $n\in\N,k\in[n]$ and $(v_0,\dots,v_k)\in[n]^{k+1}_{\neq}$,
	\[
	\cI_n(v_0,\dots,v_k) = \{v_k\text{ has no neighbour in }[n]\setminus\{v_{k-1},v_k\}\}.
	\]
	Then $\cI_n(v_0,\dots,v_k)$ satisfies \eqref{eq:lem1_as3} and \eqref{eq:lem1_as1} since the path $v_0\hdots v_k$ ensures that $v_k$ already has one neighbour. Additionally, let $g(x,k)=e^{-x}$ and $p=1$. For \eqref{eq:lem2_bound_3} we calculate 
	\begin{align*}
		&\bigg|\p_\cW(\cI_n(v_0,\dots,v_k))-e^{-W_{v_k}}\bigg|=\bigg|\E_\cW\bigg[\prod_{x\in[n]\setminus\{v_{k-1},v_k\}}\one\{v_k\nleftrightarrow x\}\bigg]-e^{-W_{v_k}}\bigg|\\
		&=\bigg|\exp\bigg(-W_{v_k}\frac{\sum_{x\in[n]\setminus\{v_{k-1},v_k\}}W_x}{L_n}\bigg)-e^{-W_{v_k}}\bigg|\\
		&\leq e^{-W_{v_k}}\bigg(e^{2W_{(n)}/L_n}-1\bigg)\leq e^{2W_{(n)}/L_n}-1\eqqcolon R_n
	\end{align*}
	so that \eqref{eq:Rn_small}  holds due to \eqref{eq:Wn_Ln}.
	
	For vertices of degree $m\geq 2$ we proceed similarly. We define
	\[
	\cI_n(v_0,\dots,v_k) =\{v_k\text{ has exactly }m-1 \text{ neighbours in }[n]\setminus \{v_{k-1},v_k\}\}
	\]
	and see that \eqref{eq:lem1_as3} and \eqref{eq:lem1_as1} hold. We rewrite $\one_{\cI_n(v_0,\dots,v_k)}$ as
	\begin{align*}
		\frac{1}{(m-1)!}\hspace{-5pt}\sum_{(a_1,\dots,a_{m-1})\in([n]\setminus\{v_{k-1},v_k\})^{m-1}_{\neq}}\prod_{i=1}^{m-1}\one\{v_k\leftrightarrow a_i\}\hspace{-15pt}\prod_{x\in[n]\setminus\{a_1,\dots,a_{m-1},v_{k-1},v_k\}}\hspace{-15pt}\one\{v_k\nleftrightarrow x\}
	\end{align*}
	and define $g(x,k)=x^{m-1}e^{-x}/(m-1)!$, which is a bounded function as required. For \eqref{eq:lem2_bound_3} we provide upper and lower bounds for $\p_\cW(\cI_n(v_0,\dots,v_k))$. We compute with the equality above and Lemma \ref{lem:prob_bound_connection}
	\begin{align*}
		&\p_\cW(\cI_n(v_0,\dots,v_k))\\
		&\leq \frac{1}{(m-1)!}\hspace{-5pt}\sum_{a_1,\dots,a_{m-1}=1}^n\prod_{i=1}^{m-1}\frac{W_{a_i}W_{v_k}}{L_n}\exp\bigg(-W_{v_k}\hspace{-20pt}\sum_{x\in[n]\setminus\{a_1,\dots,a_{m-1},v_{k-1},v_k\}}\frac{W_{x}}{L_n}\bigg)\\
		&\leq \frac{W_{v_k}^{m-1}e^{-W_{v_k}}}{(m-1)!}\hspace{-5pt}\sum_{a_1,\dots,a_{m-1}=1}^n\prod_{i=1}^{m-1}\frac{W_{a_i}}{L_n}\exp\bigg(\frac{(m+1)W_{(n)}^2}{L_n}\bigg)\\
		&=g(W_{v_k},k)\exp\bigg(\frac{(m+1)W_{(n)}^2}{L_n}\bigg)\eqqcolon g(W_{v_k},k)\overline{I}_n,
	\end{align*}
	where the upper bound $\overline{I}_n$ satisfies $\overline{I}_n\overset{\p}{\longrightarrow}1$ as $n\to\infty$ by \eqref{eq:Wn_Ln}.	Moreover, by Lemma \ref{lem:prob_bound_connection}
	\begin{align*}
		&\p_\cW(\cI_n(v_0,\dots,v_k))\geq \frac{W_{v_k}^{m-1}e^{-W{v_k}}}{(m-1)!}\hspace{-25pt}\sum_{(a_1,\dots,a_{m-1})\in([n]\setminus\{v_{k-1},v_k\})^{m-1}_{\neq}}\prod_{i=1}^{m-1}\frac{W_{a_i}}{L_n}\bigg(1-\frac{W_{(n)}^2}{L_n}\bigg)^{m-1}\\
		&=  g(W_{v_k},k)\bigg(1-\frac{W_{(n)}^2}{L_n}\bigg)^{m-1}\bigg(1-\hspace{-5pt}\sum_{(a_1,\dots,a_{m-1})\in[n]^{m-1}\setminus([n]\setminus\{v_{k-1},v_k\})^{m-1}_{\neq}}\prod_{i=1}^{m-1}\frac{W_{a_i}}{L_n}\bigg).
	\end{align*}
	The last sum above involves all tuples $a$ of length $m-1$ which either contain the same entry twice or contain at least one entry equal to $v_{k-1}$ or $v_{k}$. This leads to
	\begin{align*}
		\sum_{(a_1,\dots,a_{m-1})\in[n]^{m-1}\setminus([n]\setminus\{v_{k-1},v_k\})^{m-1}_{\neq}}\prod_{i=1}^{m-1}\frac{W_{a_i}}{L_n}\leq \frac{W_{(n)}}{L_n}\bigg(\binom{m}{2}+2m\bigg)\leq 3m^2\frac{W_{(n)}}{L_n}
	\end{align*}
	so that 
	\begin{align*}
		\p_\cW(\cI_n(v_0,\dots,v_k))&\geq g(W_{v_k},k)\bigg(1-\frac{W_{(n)}^2}{L_n}\bigg)^{m-1}\bigg(1-3m^2\frac{W_{(n)}}{L_n}\bigg)\eqqcolon g(W_{v_k},k)\underline{I}_n,
	\end{align*}
	where the lower bound $\underline{I}_n$ satisfies $\underline{I}_n\overset{\p}{\longrightarrow}1$ as $n\to\infty$ by \eqref{eq:Wn_Ln}.	For $k\in[n],p=\sup_{x\in[0,\infty)}g(x,k)$ we obtain 
	\begin{align*}
		\big|\p_\cW(\cI_n(v_0,\dots,v_k))-g(W_{v_k},k)\big|& \leq g(W_{v_k},k)(|1-\overline{I}_n|+|1-\underline{I}_n|)\\
		&\leq p (|1-\overline{I}_n|+|1-\underline{I}_n|)\eqqcolon pR_n,
	\end{align*}
	which shows \eqref{eq:lem2_bound_3}. From $\underline{I}_n,\overline{I}_n \overset{\p}{\longrightarrow}1$ as $n\to\infty$ it follows $R_n\overset{\p}{\longrightarrow}0$ as $n\to\infty$ as demanded in \eqref{eq:Rn_small}.
\end{proof}

\begin{proof}[Proof that \textbf{terminal trees} satisfy (A1)]
	We consider a tree $T$ with $m$ vertices $1,\dots,m$ and root $1$. We write $V(T)=[m]$ for its vertices and $E(T)$ for its edges. For $n\in\N,k\in[n]$ and $(v_0,\dots,v_k)\in[n]^{k+1}_{\neq}$ let
	\begin{align*}
		\cI_n(v_0,\dots,v_k) = \{ &\exists (a_1,\dots,a_m)\in([n]\setminus\{v_0,\dots,v_{k-1}\})^m_{\neq}\text{ with } a_1=v_k \text{ such that the}\\
		&\text{ graph induced by }a_1,\dots,a_m\text{ with $a_1$ as root is isomorphic to $T$}\},	
	\end{align*}
	$p(k)=k+1$ and with  $c(T)$ as in Theorem \ref{Lem:assump_for_different_vertices}
	\begin{align*}
		g(x,k)=\frac{x^{\deg_T(1)}e^{-x}}{c(T)}\prod_{i=2}^m\frac{\E[W^{\deg_T(i)}e^{-W}]}{\E[W]}.
	\end{align*}
	We note that \eqref{eq:lem1_as3} and \eqref{eq:lem1_as1} are satisfied by the choice of $\cI_n(v_0,\dots,v_k)$. We write $\varphi\colon V(T)\to[n]$ for the map $i\mapsto a_i$ which assigns a vertex of $T$ to its corresponding vertex in $G_n$. Consequently, $\varphi(E(T))=\{\{\varphi(i),\varphi(j)\}\colon i,j\in V(T)\}$ denotes the edges of $T$ embedded into $G_n$ via $\varphi$. For $\varphi(V(T))$ to be isomorphic to $T$, we require all edges in $\varphi(E(T))$ to exist in $G_n$ and no further connections between any vertices in $\varphi(V(T))$ are allowed. Furthermore, all edges between $\varphi(V(T))$ and $[n]\setminus\varphi(V(T))$ are forbidden since the tree is supposed to be terminal, up to the exception of $a_1=v_k$ being connected to $v_{k-1}$. We gather all these forbidden edges in the set 
	\begin{align*}
		F_a=\{\{a_x,y\}\colon x\in V(T),y\in[n]\setminus\{x\}\}\setminus ( \varphi(E(T))\cup\{v_k,v_{k-1}\}).
	\end{align*}
	Here we do not have $\{x,x\}\in F_a$ for any $x\in V(T)$ since we chose the convention to ignore loops. Considering all possible choices for $a_1,\dots,a_m$ and permutations thereof yields
	\begin{align}
		\one_{\cI_n(v_0,\dots,v_k)}=c(T)^{-1}\hspace{-35pt}\sum_{\underset{a_1=v_k}{(a_1,\dots,a_m)\in([n]\setminus\{v_0,\dots,v_{k-1}\})^{m}_{\neq}}}\hspace{5pt}\prod_{\{i,j\}\in E(T)}\one\{a_i\leftrightarrow a_j\} \prod_{\{i,j\}\in F_a}\one\{i\nleftrightarrow j\}. \label{eq:terminal_trees_indicator}
	\end{align}
	All occurring factors are independent when conditioning on $\cW$. For \eqref{eq:lem2_bound_3} we provide upper and lower bounds for $\p_\cW(\cI_n(v_0,\dots,v_k))$. We start with the last product concerning the forbidden edges. Let $a_1=v_k$ and $a_2,\dots,a_m\in([n]\setminus \{v_0,\dots,v_{k}\})^{k-1}_{\neq}$. As $E_n\{x,y\}$ follows a Poisson distribution, we get
	\begin{align*}
		\E_\cW\bigg[\prod_{\{i,j\}\in F_a}\one\{i \nleftrightarrow j\}\bigg]=\prod_{\{i,j\}\in F_a}\exp\bigg(-\frac{W_iW_j}{L_n}\bigg).
	\end{align*}
	Since all factors are bounded by one, we may add similar factors for a lower bound. Since $F_a\subseteq\{\{a_x,y\}\colon x\in V(T),y\in[n]\}$ this yields
	\begin{align}
		\E_\cW\bigg[\prod_{\{i,j\}\in F_a}\one\{i \nleftrightarrow j\}\bigg]\geq \prod_{x=1}^m\prod_{y=1}^n\exp\bigg(-\frac{W_{a_x}W_y}{L_n}\bigg)=\prod_{x=1}^m\exp(-W_{a_x}). \label{eq:term_trees_lb1}
	\end{align}
	For an upper bound we observe that $F_a\supseteq \{\{a_x,y\}\colon x\in V(T),y\in[n]\setminus (\varphi(V(T))\cup \{v_{k-1},x\})\}$. In comparison to $\{\{a_x,y\}\colon x\in V(T),y\in[n]\}$, there are at most $m(m+2)$ fewer elements. Therefore,
	\begin{align}
		&\E_\cW\bigg[\prod_{\{i,j\}\in F_a}\one\{i \nleftrightarrow j\}\bigg]\leq \prod_{x=1}^m\prod_{y=1}^n\exp\bigg(-\frac{W_{a_x}W_y}{L_n}\bigg)\exp\bigg(m(m+2)\frac{W_{(n)}^2}{L_n}\bigg)\nonumber\\
		&=\prod_{x=1}^m\exp(-W_{a_x})\exp\bigg(m(m+2)\frac{W_{(n)}^2}{L_n}\bigg)\eqqcolon \prod_{x=1}^m\exp(-W_{a_x}) \overline{I}_n,\label{eq:term_trees_ub1}
	\end{align}
	where $\overline{I}_n$ converges in probability to $1$ as $n\to\infty$ due to \eqref{eq:Wn_Ln}. Next, we consider the first product in \eqref{eq:terminal_trees_indicator}. From Lemma \ref{lem:prob_bound_connection} and the fact that a tree with $m$ vertices has $m-1$ edges we get the upper bound
	\begin{align}
		\E_\cW\bigg[\prod_{\{i,j\}\in E(T)}\one\{a_i\leftrightarrow a_j\}\bigg]\leq \prod_{\{i,j\}\in E(T)} \frac{W_{a_i}W_{a_j}}{L_n}=W_{a_1}^{\deg_T(1)}\prod_{i=2}^m\frac{W_{a_i}^{\deg_T(i)}}{L_n} \label{eq:term_trees_ub2}
	\end{align}
	and similarly the lower bound 
	\begin{align}
		\E_\cW\bigg[\prod_{\{i,j\}\in E(T)}\one\{a_i\leftrightarrow a_j\}\bigg]&\geq W_{a_1}^{\deg_T(1)}\prod_{i=2}^m\frac{W_{a_i}^{\deg_T(i)}}{L_n} \times \bigg(1-\min(1,W_{(n)}^2/L_n)\bigg)^{m-1}\nonumber\\
		&\eqqcolon  W_{a_1}^{\deg_T(1)}\prod_{i=2}^m\frac{W_{a_i}^{\deg_T(i)}}{L_n}  \underline{I}_n,  \label{eq:term_trees_lb2}
	\end{align}
	where $\underline{I}_n$ converges in probability to $1$ as $n\to\infty$ by \eqref{eq:Wn_Ln}. Using the upper bounds \eqref{eq:term_trees_ub1} and \eqref{eq:term_trees_ub2} for the conditional expectation of \eqref{eq:terminal_trees_indicator} yields
	\begin{align*}
		&\p_\cW(\cI_n(v_0,\dots,v_k))\leq c(T)^{-1}\hspace{-35pt}\sum_{\underset{a_1=v_k}{(a_1,\dots,a_m)\in([n]\setminus\{v_0,\dots,v_{k-1}\})^{m}_{\neq}}}\hspace{-15pt} W_{a_1}^{\deg_T(1)}\prod_{i=2}^m\frac{W_{a_i}^{\deg_T(i)}}{L_n}  \prod_{x=1}^m\exp(-W_{a_x}) \overline{I}_n\\
		&\leq \frac{W_{v_k}^{\deg_T(1)}e^{-W_{v_k}}}{c(T)}\prod_{i=2}^m \frac{\sum_{a_i=1}^nW_{a_i}^{\deg_T(i)}e^{-W_{a_i}}}{L_n} \overline{I}_n \eqqcolon \frac{W_{v_k}^{\deg_T(1)}e^{-W_{v_k}}}{c(T)}X_n\overline{I}_n,
	\end{align*}
	where the weak law of large numbers gives us
	\begin{align*}
		X_n\overset{\p}{\longrightarrow} \prod_{i=2}^m\frac{\E[W^{\deg_T(i)}e^{-W}]}{\E[W]}\eqqcolon X \quad\text{as}\quad n\to\infty.
	\end{align*}
	The lower bounds \eqref{eq:term_trees_lb1} and \eqref{eq:term_trees_lb2} provide
	\begin{align*}
		&\p_\cW(\cI_n(v_0,\dots,v_k))\geq c(T)^{-1}\hspace{-35pt}\sum_{\underset{a_1=v_k}{(a_1,\dots,a_m)\in([n]\setminus\{v_0,\dots,v_{k-1}\})^{m}_{\neq}}}\hspace{-15pt} W_{a_1}^{\deg_T(1)}\prod_{i=2}^m\frac{W_{a_i}^{\deg_T(i)}}{L_n}  \prod_{x=1}^m\exp(-W_{a_x}) \underline{I}_n\\
		&=\frac{W_{v_k}^{\deg_T(1)}e^{-W_{v_k}}}{c(T)}\underline{I}_n\bigg(X_n -\hspace{-5pt}\sum_{(a_2,\dots,a_m)\in[n]^{m-1}\setminus([n]\setminus\{v_0,\dots,v_k\})^{m-1}_{\neq}} \prod_{i=2}^m\frac{W_{a_i}^{\deg_T(i)}e^{-W_{a_x}}}{L_n}\bigg).
	\end{align*}
	In the last term we sum over all tuples $(a_2,\dots,a_m)$ which contain two equal entries or one entry from the set $\{v_0,\dots,v_k\}$. The number of such tuples is bounded by $$\binom{m-1}{2}n^{m-2}+(k+1)(m-1)n^{m-2},$$ where the first summand accounts for two equal entries whereas the second one considers the case where one entry lies in $\{v_0,\dots,v_k\}$. Noting that $C\coloneqq \sup_{x\in[0,\infty)}x^me^{-x}<\infty$ as well as $\underline{I}_n\leq1, \deg_T(i)\leq m$ we obtain
	\begin{align*}
		&\frac{W_{v_k}^{\deg_T(1)}e^{-W_{v_k}}}{c(T)}\underline{I}_n\hspace{-10pt}\sum_{(a_2,\dots,a_m)\in[n]^{m-1}\setminus([n]\setminus\{v_0,\dots,v_k\})^{m-1}_{\neq}}\prod_{i=2}^m\frac{W_{a_i}^{\deg_T(i)}e^{-W_{a_i}}}{L_n} \\
		&\leq \frac{C^{m}}{c(T)} \frac{n^{m-2}}{L_n^{m-1}}\bigg(\binom{m-1}{2}+(k+1)(m-1)\bigg)\\
		&\leq (k+1)\frac{C^{m}}{c(T)}\frac{n^{m-2}}{L_n^{m-1}}\bigg({\binom{m-1}{2}+(m-1)}\bigg)\eqqcolon (k+1)J_n= p(k)J_n,
	\end{align*}
	where $J_n$ converges in probability to $0$ as $n\to\infty$. The lower and upper bound provide
	\begin{align*}
		&\bigg|\p_\cW(\cI_n(v_0,\dots,v_k))-g(W_{v_k},k)\bigg|\\
		&\leq p(k)J_n+\frac{W_{v_k}^{\deg_T(1)}e^{-W_{v_k}}}{c(T)}\bigg(|X-X_n\overline{I}_n|+|X-X_n\underline{I}_n|\bigg)\\
		&\leq p(k)\bigg(J_n+\frac{C}{c(T)}\bigg(|X-X_n\overline{I}_n|+|X-X_n\underline{I}_n|\bigg)\bigg)\eqqcolon p(k)R_n,
	\end{align*}
	which is condition \eqref{eq:lem2_bound_3}. Since $J_n,X_n,\overline{I}_n$ and $\underline{I}_n$ do not depend on $k$ and converge in probability to $0,X,1$ and $1$ as $n\to\infty$, respectively, we conclude that $R_n\overset{\p}{\longrightarrow}0$ as $n\to\infty$, showing \eqref{eq:Rn_small} and finishing the proof.
\end{proof}

\end{document}